\documentclass[11pt, reqno]{amsart}
 
\usepackage[cp1251]{inputenc}
\usepackage{amsmath,amsthm,mathrsfs}
\usepackage{amssymb}
\usepackage[english]{babel}
\usepackage{esint}
\usepackage[titletoc]{appendix}

\usepackage{graphicx}
\usepackage{float}
\usepackage{subfigure}

\usepackage{bbm}
\usepackage{amsfonts,amstext,amssymb,verbatim,epsfig}%,psfrag}
\usepackage{pgf,tikz}
\usepackage{float}
\usepackage{enumerate}
\usetikzlibrary{arrows}
\usepackage{hyperref}
\usepackage{cite}
\usepackage{epstopdf}
\usepackage{color}
\usepackage{transparent}
\usepackage{subfigure}
\usepackage{pgfplots}
\usepgfplotslibrary{polar}
\usepgflibrary{shapes.geometric}
\usetikzlibrary{calc}

\usepackage{tikz}

\usepackage{xcolor}
\usepackage[normalem]{ulem}

\hypersetup{colorlinks=true,pdfborder=001,  citecolor  = blue, citebordercolor= {magenta}}
\hypersetup{linkcolor=magenta, linkbordercolor = blue}
\hypersetup{colorlinks,urlcolor=blue}
\makeatletter
\let\reftagform@=\tagform@
\def\tagform@#1{\maketag@@@{(\ignorespaces\textcolor{magenta}{#1}\unskip\@@italiccorr)}}
\renewcommand{\eqref}[1]{\textup{\reftagform@{\ref{#1}}}}
\makeatother

\makeatletter
\DeclareUrlCommand\ULurl@@{%
  \def\UrlLeft{\uline\bgroup}%
  \def\UrlRight{\egroup}}
\def\ULurl@#1{\hyper@linkurl{\ULurl@@{#1}}{#1}}
\DeclareRobustCommand*\ULurl{\hyper@normalise\ULurl@}
\makeatother

\def\lessim{\ \lower4pt\hbox{$
		\buildrel{\displaystyle <}\over\sim$}\ }
\def\gessim{\ \lower4pt\hbox{$\buildrel{\displaystyle >}
		\over\sim$}\ }

\newtheorem{theorem}{\bf Theorem}
\newtheorem{lemma}[theorem]{\bf Lemma}
\newtheorem{corollary}[theorem]{\bf Corollary}

\theoremstyle{remark}

\newenvironment{Proof of lemma}{\noindent{\bf Proof of Lemma}}{\hfill$\Box$\newline}
\newenvironment{Proof of theorem}{\noindent{\bf Proof of Theorem}}{\hfill{\footnotesize${\square}$}\newline}
\newenvironment{Proof of theorems}{\noindent{\bf Proof of Theorems}}{\hfill$\Box$\newline}
\newenvironment{Proof of proposition}{\noindent{\bf Proof of Proposition}}{\hfill$\Box$\newline}
\newenvironment{Proof of propositions}{\noindent{\bf Proof of Propositions}}{\hfill$\Box$\newline}
\newenvironment{Proof of exercise}{\noindent{\it Proof of Exercise:}}{\hfill$\Box$}
\begin{document}
\title{Existence of Full Replica Symmetry Breaking for the Sherrington-Kirkpatrick Model at Low Temperature}

%%
%% If there is another author uncomment and edit the following.
%%

\author{Yuxin Zhou}
\address{Department of Statistics, University of Chicago}
\email{yuxinzhou@uchicago.edu}

\begin{abstract}
We verify the existence of full replica symmetry breaking (FRSB) for the Sherrington-Kirkpatrick (SK) model. %and determine the structure of its Parisi measure slightly beyond the high temperature regime.  %It's well known that the Parisi measure of the SK model is replica symmetric at high temperature, i.e. $0<\beta\leq \frac{1}{\sqrt{2}}.$ Our main results 
{\color{black}More specifically, we prove that slightly beyond the critical temperature, the Parisi measure for the SK model is supported on an interval starting at the origin and only has one jump discontinuity at the right endpoint.}%consists of an interval starting at the origin slightly beyond the high temperature regime. 
%{(regime? Also same question for other places)} %: there exists $\eta>0$ such that for any $\frac{1}{\sqrt{2}}<\beta<\frac{1}{\sqrt{2}}+\eta$, there exists $\upsilon_\beta>0$ such that the interval $[0,\upsilon_\beta]$ is in  the support of the Parisi measure.
%there exists $\eta>0$ such that the SK model is FRSB for $\frac{1}{\sqrt{2}} < \beta < \frac{1}{\sqrt{2}}+\eta$.
\end{abstract}

\maketitle
\section{Introduction and main results}

The Sherrington-Kirkpatrick (SK) model is a crucial example of mean field spin glasses, leading to a wide range of problems and phenomena in both the physical and mathematical sciences. For detailed information on its background, history, and methods, we direct the reader's attention to the books by Mezard, Parisi, and Virasoro \cite{mezard1987spin}, as well as Talagrand \cite{talagrand2006parisi-b} and their extensive references.

In this paper, we investigate the structure of the functional order parameter for the Sherrington-Kirkpatrick(SK) model. This order parameter, referred to as the Parisi measure, is expected to provide a comprehensive qualitative description of the system and has been extensively studied by researchers in both physics and mathematics \cite{mezard1987spin, talagrand2006parisi-b}.  Recent discoveries have shed light on Parisi measures in \cite{AChen15PTRF,Aukosh17PTRF}, yet the structure of these measures remains elusive at low temperature. {The purpose of the present paper is to rigorously establish a key property of the Parisi measure known as ``full replica symmetry breaking'' and determine its structure when the temperature drops below a threshold.}
%In previous studies of mean field Ising spin glasses, only  Parisi measures that are replica symmetric have been rigorously proven to exist in specific models  and no rigorous examples  beyond the replica symmetric phase have been established.

%In all prior works, only  Parisi measures that are replica symmetric have been rigorously proven to exist in certain models and  no rigorous examples of mean field Ising spin glass beyond replica symmetric phase were established.

\subsection{{Background: The Ising spin glass model and Parisi measures}} We first  introduce the mean field Ising spin glass models. 
Let $p,N$ be integers with $p \geq 2$ and $N\geq1$. For any $N \geq 1$,  let $\Sigma_N:= \{ -1,+1 \}^N$ be the Ising spin configuration space. The Hamiltonian of the mean field Ising pure $p$-spin model  is a Gaussian function defined as
\begin{eqnarray*}
H_{N,p}(\sigma):=\frac{1}{N^{\frac{p-1}{2}}} \sum_{1 \leq i_1,\cdots,i_p \leq N} g_{i_1,\dots,i_p} \sigma_{i_1} \cdots \sigma_{i_p},
\end{eqnarray*}
for $\sigma=(\sigma_1,\cdots,\sigma_N) \in \Sigma_N,$
where all $(g_{i_1,\cdots, i_p})$, $1 \leq i_1,\cdots i_{p} \leq N$,  are independent, identically distributed standard Gaussian random variables.  
%Here $\beta>0$ is a parameter of the Ising pure $p$-spin model, called the inverse temperature.  When $p=2,$ the model is the well-known SK model. 

More generally, {one can} also consider the Ising mixed $p$-spin  model defined on $\Sigma_N$ whose Hamiltonian is a linear combination of the pure $p$-spin Hamiltonians 
%given a sequence of non-negative numbers $\gamma_p$ such that $\sum^\infty \gamma^2_p (1+\varepsilon)^p < \infty$ for $\varepsilon>0$ sufficiently small, the mixed $p$-spin Hamiltonian is defined as
\begin{eqnarray*}
H_N(\sigma)=\sum^\infty_{p=2} \beta_p H_{N,p}(\sigma),
\end{eqnarray*}
where $H_{N,p}$'s are assumed to be independent for different values of $p$. Here the sequence $\boldsymbol{\beta}:= (\beta_p)_{p \geq 2}$ is called the temperature parameters satisfying that $\sum^\infty_{p=2} 2^p \beta^2_p <\infty.$ 
%As for the pure $p$-spin model, i.e. $H_N= \beta_p H_{N,p}$, we will simply use $\beta>0$ rather than $\beta_p$ to denote the temperature parameter. 

The Gaussian field $H_{N}$ is centered with covariance given by 
\begin{eqnarray*}
\mathbb{E} H_N(\sigma^1) H_N(\sigma^2)=N \xi(R_{1,2})
\end{eqnarray*}
where  $R_{1,2}:=\frac{1}{N}\sum^N_{i=1} \sigma_i^1 \sigma^2_i$
is the normalized inner product between $\sigma^1$ and $\sigma^2$ and 
\begin{eqnarray}\label{eq:psxi}
\xi(x):= \sum_{p \geq 2} \beta^2_p x^p.
\end{eqnarray}
When $\xi(x)= \beta_2^2 x^2,$ the model introduced above is the well-known SK model, which is a mean field modification of the Edwards-Anderson model \cite{EA}.

One of the most important problems in the Ising spin glass model introduced above is to compute  the ground state energy
$$\max_{\sigma\in\Sigma_N} H_N(\sigma),$$
and 
the ground state
$${\arg\max}_{\sigma\in\Sigma_N} H_N(\sigma),$$
 as $N$ tends to infinity, which is indeed an extremely challanging task. One standard approach in statistical mechanics is to consider the Gibbs measure of $H_N$
\begin{eqnarray*}
G_{N,\beta}(\sigma)=\frac{1}{Z_{N}} \exp  H_N(\sigma)
\end{eqnarray*}
and the corresponding free energy
\begin{eqnarray*}
F_{N,\beta}=\frac{1}{ N} \log Z_{N,\beta},
\end{eqnarray*}
where $Z_{N,\beta}$ is the partition function of $H_N$ defined as
\begin{eqnarray*}
Z_{N,\beta}=\sum_{\sigma \in \Sigma_N} \exp   H_N(\sigma).
\end{eqnarray*}
 The central goal in this approach is to describe the limiting free energies $F_{N,\beta}$ and  Gibbs measures $G_{N,\beta}$ as $N$ tends to infinity at different values of $\boldsymbol{\beta}$. 

A groundbreaking  solution to the limiting free energy of the SK model was proposed by Parisi  \cite{parisi1979infinite,parisi1980infinite}, where it was predicted that the thermodynamic limit of the free energy
can be computed using a variational formula.%, known as the Parisi formula. 
    This formula, known as the Parisi formula was later rigorously validated and extended to all mixed $p$-spin models by  Panchenko and Talagrand\cite{panchenko, talagrand2006parisi}. 
  To be more specific, denote the space  of all probability measures on $[0,1]$ by $M[0,1]$ and the support of $\mu \in M[0,1]$ by supp{\color{white}.}$\mu$.  For any $\boldsymbol{\beta}=(\beta_p)_{p \geq 2}$ and $\mu \in M[0,1]$, the Parisi functional  $\mathcal{P}_{\boldsymbol{\beta}}(\mu)$ is defined as  
\begin{eqnarray} \label{Parisifunctional}	
\mathcal{P}_{\boldsymbol{\beta}}(\mu)= \log 2 + \Phi_\mu (0,0)-\frac{1}{2} \int^1_0 \alpha_\mu(s) s \xi''(s) ds,
\end{eqnarray}
where $\Phi_\mu$ is the weak solution to the Parisi PDE on $ \mathbb{R} \times [0,1]$
\begin{equation}\label{ParisiPDE}  \left\{
\begin{array}{lcl}
\partial_u \Phi_\mu(x,u)=-\frac{\xi''(u)}{2} \Big[ \partial_{xx} \Phi_\mu(x,u) +\alpha_\mu(u) \big( \partial_x  \Phi_\mu(x,u)  \big)^2 \Big] .   \\
\Phi_\mu(x,1)=\log \cosh x .
\end{array} \right. \end{equation} and $\alpha_\mu$ is the distribution function of $\mu \in M[0,1].$
The Parisi formula states that the following limit exists almost surely,
\begin{eqnarray*}
\lim_{N \rightarrow \infty} \frac1N \log {\sum_{\sigma \in \Sigma_N}  \exp   H_N(\sigma)}=\inf_{\mu \in M[0,1]} \mathcal{P}_{\boldsymbol{\beta}} (\mu).
\end{eqnarray*}

As an infinite dimensional variational formula, $\mathcal{P}_{\boldsymbol{\beta}}$ is continuous and always has a minimizer. The uniqueness of the minimizer is first proven by  Auffinger and Chen\cite{{auffinger2015parisi}}. For any temperature parameter $\boldsymbol{\beta}$, the unique minimizer of $\mathcal{P}_{\boldsymbol{\beta}}$ is called the Parisi measure, denoted by $\mu_\beta$.

It is predicted that the Parisi measure is the limiting distribution of the overlap
$R(\sigma^1,\sigma^2)$ under the measure $\mathbb{E}G^{\otimes2}_N$. Moreover, Panchenko \cite{panchenko1} established the asymptotically ultrametricity assuming the validity of the extended Ghirlanda-Guerra identities \cite{GGIdentity} which are known to be valid for the SK model with an asymptotically vanishing perturbation. These two important properties of the Gibbs measures then implies that a hierarchical clustering structure is formed by   the spin configurations under the Gibbs measure, where the number of the levels are determined by the number of points in the support
of the Parisi measure.   The Parisi measure is then a crucial component in describing both the structure of the Gibbs measure and the system's free energy. For a more detailed discussion, we refer readers to \cite{panchenko1,mezard1987spin}.

To find the ground state of the SK model, Montanari \cite{Montnari1} gives an algorithm that for any $\varepsilon>0,$ outputs $\sigma_\star \in \Sigma_N$ such that $H_N(\sigma_\star)$ is at least $(1-\epsilon)$ of the ground state energy with probability converging to one as $N\rightarrow \infty$.   {\color{black}This algorithm  works and only works under the assumption (Assumption 1 of \cite{Montnari1}) that the Parisi measure is supported on an interval starting at $0$. Our main result verifies this assumption when $\beta$ is larger than and sufficiently close to $\frac{1}{\sqrt{2}}$.\footnote{\cite{Montnari1} has a slightly different normalization and the critical temperature is $\beta = 1$ there.}}

%See [20, 26] for detailed discussion.
%The importance of (iii) lies on the fact that it indicates the phase transition of the model between high and low temperature regimes. While at high temperature the clusters contain no layers, at low temperature they begin to possess clustering structures with multiple layers. The statement (iii) advocates the existence of infinitely many layers within the clusters at low enough temperature. In a nutshell, the Parisi measure is the key ingredient of the matter that describes the structure of the Gibbs measure as well as the free energy of the system. See [20, 26] for detailed discussion

%The Parisi measure $\mu_\beta$ is essential for characterizing the mean field spin glass models.  In contrast, the structure of the Parisi measures remains mysterious at low and zero temperature. 

\subsection{Main Results}

The significance of the Parisi measure introduced above {naturally} motivates the   {\emph{classification problem}} of the structure of the Parisi measure $\mu_\beta$.  We say that the Parisi measure $\mu_\beta$ is replica symmetric (RS) if it's a Dirac measure; $k$ levels of replica symmetric breaking ($k$-RSB) if it consists of $k+1$ atoms; full replica symmetric
breaking (FRSB) if its support contains some interval.

As for the SK model with $\xi(x)=\beta^2x^2,$ it's predicted in the physics literature {(See \S \ref{morediscussion} below)} that the Parisi measure $\mu_\beta$ is FRSB for $\beta$ sufficiently large, which plays a crucial role in Parisi's original solution of the SK model.
%The results above are some progress about the Parisi measure to be $k$-RSB.  
%Parisi predicted that there exists a critical inverse temperature $\beta_c>0$ such that for any $\beta>\beta_c$, the Parisi measure is FRSB, which plays a crucial role in Parisi's original solution to the SK model.  
%expected that in the low temperature region $\frac{1}{\sqrt{2}}$, the Parisi measure of the SK model is FRSB: 
%To be more specific, for the SK model with $\xi(x)=\beta^2 x^2,$ it's expected that when $\beta > \frac{1}{\sqrt{2}},$  
Our main results below not only verify the existence of the FRSB phase  but also determine the structure of the Parisi measure  slightly beyond the high temperature regime for the SK model. In consequence, we verify that the {\color{black}algorithms} developed by Montanari {\color{black}\cite{Montnari1} and many follow-up works} can work slightly beyond the high temperature regime.
 \begin{theorem}\label{mainthm}
Suppose the SK model with $\xi(x)=\beta^2x^2$. There exists $\eta>0$ such that for any $\frac{1}{\sqrt{2}}<\beta\leq\frac{1}{\sqrt{2}}+\eta$, there exists $\upsilon_\beta\in(0,{\color{black}\eta} )$ such that supp{\color{white}.}$\mu_\beta=[0,\upsilon_\beta]$.
   \end{theorem}
   %\begin{remark}
   Based on our main results above, the Parisi measure slightly beyond the high temperature regime then has an explicit form:
   \begin{corollary}\label{coromain}
Under the assumption of Theorem \ref{mainthm}, for any $\frac{1}{\sqrt{2}}<\beta\leq\frac{1}{\sqrt{2}}+\eta$, $\mu_\beta$ has the following form:
   \begin{eqnarray*}
   \nu_\beta+(1-m)\delta_{\upsilon_\beta}.
   \end{eqnarray*}
   Here $\nu_\beta$ is a fully supported measure on $[0,\upsilon_\beta)$ with $m:=\nu_\beta([0,\upsilon_\beta))<1$ and possesses a smooth density.
   \end{corollary}

To the best of our knowledge, the existence of the FRSB phase has not been established before in the Ising spin glass models.  The theorem above is the first result validating the existence of  FRSB  and determining the structure of the Parisi measure involving FRSB in the Ising spin glass models. It's expected that the support of the Parisi measure contains an interval for any Ising spin glasses with $\boldsymbol{\beta}$ sufficiently large. {We hope our new ingredients in the proof of Theorem \ref{mainthm} can eventually lead to the full resolution of this conjecture and related problems}.

\subsection{Earlier related works}

In this section, we survey some earlier works about the Parisi measures of the mean-field Ising spin glass models in math literature.

{For} the SK model with $\xi(x)=\beta^2 x^2,$ the Parisi measures at high temperature, i.e.  $0 < \beta \leq \frac{1}{\sqrt{2}}$  are RS proven   by Aizenman, Lebowitz and
Ruelle in \cite{ALR}. As for low temperature, Toninelli \cite{Toni} showed  that the Parisi measure is not RS for $\beta>\frac{1}{\sqrt{2}}$. Auffinger and Chen \cite{AChen15PTRF} showed that slightly above the critical temperature $\beta=\frac{1}{\sqrt{2}}$, the largest number in the support of the Parisi measure is a jump discontinuity. 
   \begin{figure}[H] 
\centering 
\includegraphics[width=0.7\textwidth]{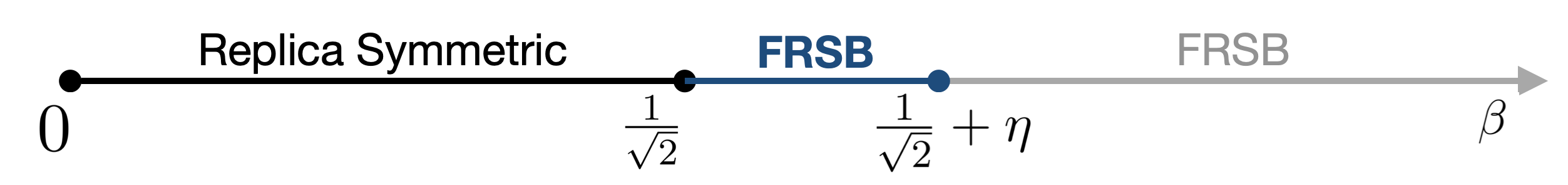} 
\caption{Phase transitions of  $\mu_{\beta}$ with respect to  $\beta$ for the SK model. The phase in black are previous results \cite{ALR} for $0<\beta\leq \frac{1}{\sqrt{2}}$. The phase in blue is our main results for $\frac{1}{\sqrt{2}}<\beta\leq\frac{1}{\sqrt{2}}+\eta$ in Theorem \ref{mainthm} and the phase in grey remains unknown.} 
\label{Fig4}
\end{figure}
Combining the previous progress above with our main results about the SK model, we have the relation between the phases of the Parisi measure $\mu_\beta$ and the temperature $\beta$, which is illustrated in Figure \ref{Fig4}. The phase in black represents  that the Parisi measure $\mu_\beta$ is RS for  $0< \beta \leq \frac{1}{\sqrt{2}}$  \cite{ALR}.  The phase in blue represents  our main results that the Parisi measure is FRSB for  $\frac{1}{\sqrt{2}}<\beta\leq\frac{1}{\sqrt{2}}+\eta$. The phase in grey is  conjectured  to be FRSB as well for $\beta> \frac{1}{\sqrt{2}}+\eta$.

 For the pure $p$-spin models with $p \geq 3,$ it was proven by Chen, Handschy and Lerman  in \cite{Chen,ChenHL} that the Parisi measure remains RS at high temperature and leave the RS phase when the temperature decreases. Recently, the author \cite{Zhou} verified the existence of 1RSB and proved that the Parisi measure is 1RSB slightly beyond the high temperature regime. 
 
 As for the mixed $p$-spin models, it was shown by Auffinger and Chen \cite{AChen15PTRF}  that the support of the Parisi measures contains the origin at all temperatures. If the support contains an open interval, then the Parisi measure has a smooth density on this interval. They also gave a criterion on temperature parameters for the Parisi measures to be neither RS nor 1RSB. Moreover, it was shown by Auffinger, Chen and Zeng \cite{auffinger2017sk} that the support of the Parisi measure contains infinitely many points at zero temperature.

%Beyond that,  there were no rigorous examples of Ising spin glass models other than the replica symmetric phase. 

\subsection{More discussion about FRSB phase of the SK model}\label{morediscussion}

In order to introduce more details about our main results,
we {first} introduce some predictions about the SK model in physics literature. {See \cite{BinderYoung} for a good review of the earlier history.} 

For the SK model with $\xi(x)= \beta^2 x^2$, the Parisi measure is expected to be FRSB of the form mentioned in Corollary \ref{coromain}
\begin{eqnarray}\label{eqfrsb}
\mu_\beta=\nu_\beta+(1-m) \delta_{\upsilon_\beta},
\end{eqnarray}
for any $\beta> \frac{1}{\sqrt{2}}$.
Here $\nu_\beta$ is a fully supported measure on $[0,\upsilon_\beta)$ with $m= \nu_\beta([0,\upsilon_\beta))<1$ and has a smooth density. {As mentioned before, prior to Theorem \ref{mainthm}, this FRSB phenomenon was not rigorously established for any $\beta>\frac{1}{\sqrt{2}}$, indeed nor for any Ising spin glass model. Next we explain some consequences of Theorem \ref{mainthm}.} 
For $0<\beta \leq \frac{1}{\sqrt{2}}+\eta$, the distribution of the Parisi measure $\mu_\beta$ is illustrated in Figure \ref{Fig3}. The figure on the left illustrates that $\mu_\beta$ is RS for $0< \beta \leq \frac{1}{\sqrt{2}}$  \cite{ALR}. The figure on the right represents our main results {(Theorem \ref{mainthm} and Corollary \ref{coromain})} for $\mu_\beta$ when $\frac{1}{\sqrt{2}}<\beta\leq \frac{1}{\sqrt{2}}+\eta.$ To be more specific,   the blue curve represents   that  the support of $\mu_\beta$ consist of $[0,\upsilon_\beta]$ for some $\upsilon_\beta\in (0,{\color{black}\eta})$ and the density of $\mu_\beta$ on $[0,\upsilon_\beta)$ is smooth.  The blue line represents that the right endpoint  of $[0,\upsilon_\beta]$ is the only jump discontinuity in  supp{\color{white}.}$\mu_\beta$. 
%Yeah! The grey line represents the prediction \eqref{eqfrsb} in the physics literature, i.e. $q=\nu_\beta$, which still remains mysterious.

 \begin{figure}[H] 
\centering 
\includegraphics[width=1\textwidth]{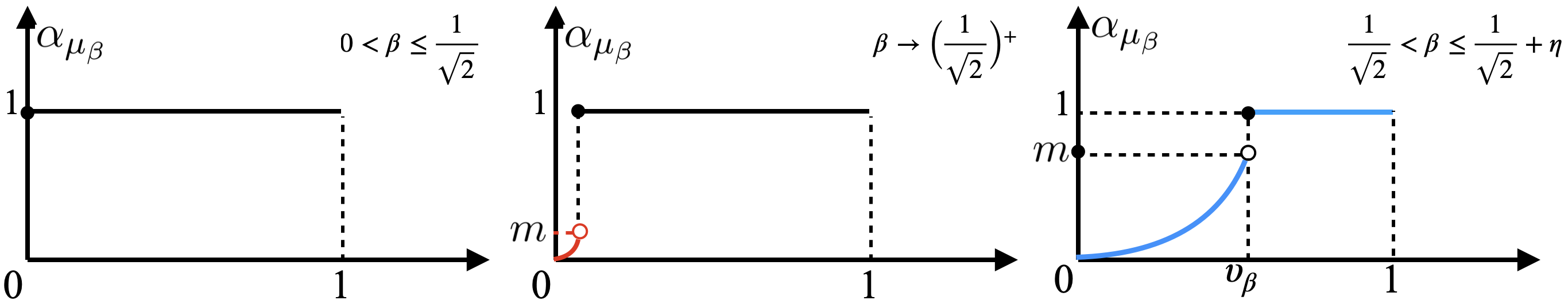} 
\caption{Distributions $\alpha_{\mu_\beta}$ of  Parisi measures $\mu_\beta$ for the SK model. The  figures from left to right are $\alpha_{\mu_\beta}$ for $0< \beta \leq \frac{1}{\sqrt{2}}$, $\beta \rightarrow \big( \frac{1}{\sqrt{2}}\big)^+$ and $\frac{1}{\sqrt{2}}< \beta \leq \frac{1}{\sqrt{2}}+ \eta$, respectively.} 
\label{Fig3}
\end{figure}
 
{One can intuitively understand} the phase transition at the critical temperature $\beta=\frac{1}{\sqrt{2}}$ {as follows:} When $0 < \beta \leq \frac{1}{\sqrt{2}}$, the Parisi measure $\mu_\beta$ is RS and  then on the verge of transitioning from RS to FRSB phase later at  $\beta=\frac{1}{\sqrt{2}}$. Indeed when $\beta=\frac{1}{\sqrt{2}}$, $\mu_\beta$ can be regarded as the threshold of RS  and FRSB as follows:  
   \begin{eqnarray*}
 \mu_\beta&=&\delta_0 \text{ (RS), }\\
 &=&\nu_\beta+(1-m)\delta_{q}, \text{ where } m=q=0 \text{ in \eqref{eqfrsb}} \text{ (FRSB). }
 \end{eqnarray*}
 Here the first equality is the usual way to regard $\mu_\beta$ as replica symmetric while the second equality is the way to regard it as a degenerate case of FRSB. As presented by the middle figure in Figure \ref{Fig3}, the red line represents that the support of $\mu_\beta$ is about to contain an interval near the origin at the critical temperature $\beta=\frac{1}{\sqrt{2}}$.

\subsection{{Proof ideas}} \label{proofidea}{In this subsection, we discuss some key new ideas in our verification of the FRSB. We start by recalling a useful criterion proved by Auffinger and Chen (Theorem \ref{thmcriterion} below) on whether a probability measure $\mu$ is the Parisi measure $\mu_{\beta}$. Starting from any $\mu$, one can construct an auxiliary function commonly denoted as $\Gamma_{\mu}$ such that:  $\Gamma_{\mu_{\beta}} (u) = u$ and $\Gamma_{\mu_{\beta}}' (u) \leq 1$ for $u \in$ supp{\color{white}.}$\mu_{\beta}$.

Suppose $\beta$ is close to $\frac{1}{\sqrt{2}}$. Our first new input to characterize the FRSB property is an elementary analysis fact that did not seem to be used on this characterization problem before. {\color{black}We will fix an absolute constant $\mathcal{K}>0$.} It is well-known that $\mu_{\beta} \neq \delta_0$ when $\beta > \frac{1}{\sqrt{2}}$ {\color{black} \cite{Toni} and that $0$ is in the support of $\mu_{\beta}$}  \cite{AChen15PTRF}}. {\color{black}We claim that Theorem \ref{mainthm} follows if we can show that none of the following three cases holds} ($A, B>0$ are absolute constants in Theorem \ref{thmFderiv} below):
\begin{enumerate}[(I)]
\item  {\color{black} A large interval near $0$ with its endpoints contained in supp{\color{white} .}$\mu_{\beta}$    is missing from supp{\color{white} .}$\mu_{\beta}$, i.e. $(q,q') \cap$ supp{\color{white} .}$\mu_{\beta}=\emptyset$,   $q\leq\mathcal{K}$, $q'-q>\frac{A}{3B}$ and $q,q' \in$ supp{\color{white} .}$\mu_{\beta}$.}
%Some interval with {\color{black}only} the left endpoint close to the origin is missing from supp{\color{white}.}$\mu_{\beta}$: For some $q \in [0,\mathcal{K})$ and $q' \in (q,1],$ $(q, q')\notin$ supp{\color{white}.}$\mu_{\beta}$ but $q,q' \in$ supp{\color{white}.}$\mu_{\beta}$.
\item {\color{black} A small interval near $0$ with its endpoints contained in supp{\color{white} .}$\mu_{\beta}$    is missing from supp{\color{white} .}$\mu_{\beta}$, i.e. $(q,q') \cap$ supp{\color{white} .}$\mu_{\beta}=\emptyset$,   $q\leq\mathcal{K}$, $q'-q\leq\frac{A}{3B}$ and $q,q'\in$ supp{\color{white} .}$\mu_{\beta}$.} 
%{\color{black}Some interval with both endpoints close to the origin is missing from supp{\color{white}.}$\mu_{\beta}$: For some $q, q' \in [0,\mathcal{K})$, $(q, q')\notin$ supp{\color{white}.}$\mu_{\beta}$ but $q,q' \in$ supp{\color{white}.}$\mu_{\beta}$.}
%There exists a sequence of intervals approaching to some point near the origin: For some $q \in [0,\mathcal{K}),$ there is a sequence of intervals $\{(q_{1, n}, q_{2, n})\}_{n \geq 1} \text{ with } q<q_{1, n}<q_{2, n}<q+\frac{1}{n}$ satisfying that $(q_{1, n}, q_{2, n}) \notin$ supp{\color{white}.}$\mu_{\beta}$ but $q_{1, n}, q_{2, n} \in$ supp{\color{white}.}$\mu_{\beta}$.
\item {\color{black} Some point away from $0$ is in supp{\color{white} .}$\mu_{\beta}$, i.e. $[\mathcal{K},1]\cap$ supp{\color{white} .}$\mu_{\beta}\not=\emptyset$.}
%Some interval with the left endpoint far away from the origin is missing from supp{\color{white}.}$\mu_{\beta}$: $\mu_\beta$ is within the same setting in the previous two cases  except $q \in [\mathcal{K},1].$
\end{enumerate}

{\color{black}Indeed, note that supp{\color{white} .}$\mu_{\beta}$ contains $0$ but is not equal to $\{0\}$. To see that ruling out all the three cases gives Theorem \ref{mainthm}, suppose to the contrary that supp{\color{white} .}$\mu_{\beta}$ is not an interval starting at $0$. Then it must miss an interval $(q, q')$ but contains $q, q'$. If $q\leq \mathcal{K}$, then Cases I or II must occur. Otherwise, Case III  occurs. We will show in Subsection \ref{CaseIII} that Case III cannot hold for $\beta$ close to $\frac{1}{\sqrt{2}}$. Thus if we rule out all the three cases, then supp{\color{white} .}$\mu_{\beta}$ must be  an interval starting at $0$. Moreover the length of this interval is small because Case III never happens.} %In fact, we will show all of them violate Theorem \ref{thmcriterion}.

There are two major well-known difficulties that prevented many efforts trying to use results like Theorem \ref{thmcriterion} to characterize the symmetric breaking structure of $\mu_{\beta}$. 

{\color{black}\underline{Difficulty 1.}} Even though $\Gamma_{\mu_\beta}$ is defined by a formula, it is not explicitly computable as the definition involves solving a nonlinear PDE, known as Parisi PDE, and a random process. Thus,  one must look for a {\color{black}useful} estimate  instead. As we shall see, many {\color{black}estimates} concerning  related functions to the Ising spin glasses are far from trivial and sometimes unusually tight. As a consequence, careless {\color{black}estimates} almost never work. Interested readers can also see \cite{Zhou}, where this difficulty already presents itself.

{\color{black}\underline{Difficulty 2.}} In order to establish the qualitative FRSB property, we need to rule out all possibilities for $\mu_{\beta}$ being in Cases I, II or III. That is an uncountably infinite dimensional set to rule out: we cannot parameterize all measures in Cases I, II or III with countably many parameters. It is challenging to prove none of them can be $\mu_{\beta}$ at the same time and perhaps a reason why no rigorous FRSB results existed in the study of this model before.

Next we explain how both difficulties are overcome in the present paper. {\color{black}As above, Difficulty 2 makes the FRSB characterization problem resist all attacks up to date. To possibly deal with it, we need to find something in common for the infinitely dimensional space of non-FRSB measures that disqualifies them of being the Parisi measure all at once.} We follow an important principle in our prior work \cite{Zhou} on Ising pure $p$-spin glasses ($p\geq3$), which is in turn inspired by our prior works \cite{AZhou1,Zhou2}. In \cite{Zhou}, a phase transition from RS to 1RSB near a critical temperature is established by considering an auxiliary function related to $\Gamma_{\mu}$ and proving it is convex. {\color{black}Next we will see} that another auxiliary function based on $\Gamma_{\mu}$ is crucial in our approach. It will have provable nice properties to rule out Cases I and II. Once we rule out the two cases, we can then rule out Case III by a subtle observation, which will be introduced briefly at the end of this subsection.

%(Shen me po wan 1 2)It looks like Case II may be the most troublesome, so let us think about what can go wrong with a measure of this kind. When $(q_{1, n}, q_{2, n})$ is missed from supp$\mu_{\beta}$}

It turns out the function 
%$$F_{\mu_\beta} (x) = \frac{x\cdot [\Gamma_{\mu_{\beta}}' (0) + \Gamma_{\mu_{\beta}}' (x)]}{\Gamma_{\mu_{\beta}} (x)} -2, 0<x<q$$ in Case I or the function 
$$G_{\mu_\beta} (q,t) = \frac{(t-q)\cdot [\Gamma_{\mu_{\beta}}' (q) + \Gamma_{\mu_{\beta}}' (t)]}{\Gamma_{\mu_{\beta}} (t) - \Gamma_{\mu_{\beta}} (q)} -2, q<t\leq q'$$ 
%in Case II (where $(q_1, q_2)$ is equal to some $(q_{1, n}, q_{2, n})$)
 is the one that can help us rule out Cases I and II. The motivation of this construction comes from the structure of the Parisi PDE, and it provides some unique advantages in the proof of the main theorem: On one hand, by the benchmark Theorem \ref{thmcriterion}, it holds that $G_{\mu_\beta}(q,q')\leq0$. On the other hand, we will prove that if supp{\color{white}.}$\mu_{\beta}$ leaves a small gap $(q,q')$ with $q,q'$  sufficiently close to {\color{black}$0$} (satisfied in Case II), or if it leaves a large gap $(q, q')$  but with {\color{black}$q$ close to $0$ and} $\mu_{\beta} (\{0\})$ close to $1$ (satisfied in Case I), then $G_{\mu_\beta}(q,\cdot)$  has a strictly positive limit at the right endpoint. This leads to a contradiction. We thus have the intuition that if $\mu_{\beta}$ leaves gaps near the origin and is not FRSB around it, then the gap will violate the non-positivity requirement of $G_{\mu_\beta}(q,q')$  granted by Auffinger-Chen's Theorem \ref{thmcriterion} and thus all such gaps  must be closed in the Parisi measure.

Next we explain in details of the proof of the strict positivity of $G_{\mu_\beta}(q,q')$. This property is stated and proved in Theorems \ref{thmFderiv} and \ref{Fm}. In our proofs, the assumption $\beta$ close to $\frac{1}{\sqrt{2}}$ is only used to ensure the Parisi measure $\mu_{\beta}$ is close to $\delta_0$ (in the sense of Lemma \ref{continuitylem} below).

First suppose we are in Case II so that $q$ {\color{black} is close to $0$ and the gap is small, i.e. $q \in [0,\mathcal{K}]$} and $q'-q\leq\frac{A}{3B}$. We will prove that (Theorem \ref{thmFderiv}) $G_{\mu_\beta}(q,t)$ in fact strictly increases on $(q, q']$. This is done by first deriving from the Parisi PDE that $G_{\mu_\beta}(q,q^+ ) =\partial_t G_{\mu_\beta} (q,q^+ ) =0$, and then derive {\color{black} that ${\partial_t^2}G_{\mu_\beta}(q,q^+)$ is larger than an absolute positive constant}  using the facts that ${\partial_t^2}G_{\delta_0}(0,0^+)>0$  and $G_{\mu_\beta}(q,\cdot)$ is sufficiently close to $G_{\delta_0}(0,\cdot)$ by definition. Here $G_{\delta_0}$ denotes the function defined in the same way as $G_{\mu_\beta}$ with all $\mu_{\beta}$ replaced by $\delta_0$. It is interesting to comment that our computation seems to show $\partial_t^2G_{\delta_0} (0,0^+) > 0$ because it is a sum of squares of engaging terms such as a fourth derivative of the Parisi PDE solution (see \eqref{square}). We expect this observation to be interesting in its own right and to see more applications in more FRSB characterizations.

Now suppose we are in Case I so that $(q, q')$ is {\color{black}a long} gap left out by supp{\color{white}.}$\mu_{\beta}$ {\color{black}   and its left endpoint is close to $0$, i.e. $q\in[0,\mathcal{K}]$ and $q'-q>\frac{A}{3B}$}. %Without loss of generality, we can assume this gap is at least an absolute constant (otherwise the exact proof framework in case II will also work here, recorded in Theorem \ref{thmFderiv}). 
In this case we can again approximate $G_{\mu_\beta}(q, q')$ by $G_{\delta_0} (0,q')$ and it suffices to prove $G_{\delta_0}(0,\cdot)$ is strictly positive on the whole $(0, 1]$. We will in fact prove $G_{\delta_0}(0,\cdot)$ is strictly increasing on $[0, 1]$ (Theorem \ref{Fm}), which can be evidently seen from numerical approximations and is how we discovered it is useful. But when it comes to a rigorous proof, we immediately have to face the above Difficulty 1: ${\partial_t}G_{\delta_0}(0,t)$ cannot be accurately computed as it involves expectations of Gaussian variables, but we need to bound its value everywhere{\color{black}, including at numbers faraway from $0$}. Moreover,  as with many inequalities of this model, numerics suggests this bound is very strong and special cautions need to be taken to prove it.

We prove Theorem \ref{Fm} by following a ``reverse Gaussian integration by parts'' strategy we developed in \cite{Zhou}. ${\partial_t}G_{\delta_0}(0,t)$ has the same sign as some complicated polynomial of expectations of various expressions of {\color{black}$2\beta^2 t$ and} a Gaussian random variable. It is far from linear or explicitly computable and may seem hard to control. {\color{black}A key move that was also previously used in \cite{Zhou} is to use Gaussian integration by parts in the unusual direction. By doing this and again relying on some structure of the Parisi PDE, we are able to remove the $2\beta^2 t$ terms and pin down a few deterministic linear inequalities that implies ${\partial_t}G_{\delta_0}(0,t)>0$ (\eqref{GIP1}, \eqref{GIP2} and \eqref{GIP3}). The readers will see they are correct but still surprisingly strong. %Moreover, this reverse GIBP also causes the complication that the above inequalities are not only involves unusual functions like the arc-tangent function (resulted from reverse Gaussian integration by parts) and consequently annoying to prove. 
We handle this difficulty and rigorously prove these by} %However, we pin down a few linear inequalities that implies $F_{\delta_0}'>0$, and then use Gaussian integration by parts in the unusual direction to pin down a few \emph{deterministic} inequalities (\eqref{GIP1}, \eqref{GIP2} and \eqref{GIP3}) which are correct but still surprisingly strong and involves unusual functions like the arc-tangent function (resulted from reverse Gaussian integration by parts) and consequently annoying to prove. We give a theoretical proof based on 
a very strong inequality for the inverse trigonometric functions function in the literature \cite{arcsinh}. {\color{black} We anticipate this reverse GIBP technique to see more uses in the study of Parisi measures.}

Finally we turn to Case III. Since  $\mu_{\beta}=\delta_0$ when $\beta=\frac{1}{\sqrt{2}},$ the mass of $\mu_\beta$ then concentrates near the origin when $\beta$ is close to $\frac{1}{\sqrt{2}}$, i.e. $\mu_\beta([0,\mathcal{K})) \geq 1-\mathcal{K}$. However, inspired by {\color{black}a result of Auffinger-Chen \cite[Theorem 4]{AChen15PTRF}}, we show that if there is some point far away from the origin contained in supp{\color{white}.}$\mu_\beta$, i.e. $q \in [\mathcal{K},1]$ and $q \in$ supp{\color{white}.}$\mu_\beta$, more mass than the remaining will be needed at $q$, i.e. $\mu_\beta( \{q\})>\mathcal{K}$, which leads to a contradiction.

%Since we need to consider a sequence of intervals rather than only one missing from supp $\mu_{\beta}$, Case II seems to be much more troublesome than Case I. In order to rule out both cases, it then may be wise of us to consider some common properties shared by both cases, which leads to their failure to be the Parisi measure. 

It is worth reiterating that the only advantage we take by working near the critical temperature $\beta = \frac{1}{\sqrt{2}}$ is that we can assume $\mu_{\beta}$ is near $\delta_0$, which allows some convenience in the final computational problems we pin down (\eqref{square}, \eqref{GIP1}, \eqref{GIP2} and \eqref{GIP3}). For general $\beta$, we hope some of our tools will stay useful to verify the FRSB property.

 %and the support of $\mu_\beta$ contains some interval $[0,\upsilon_\beta],$ then 

%  The distribution can also be regarded as the threshold of RS and 1RSB: 
% \begin{eqnarray*}
% \alpha_{\mu_P}(s)&=&\mathbbm{1}_{[0,1]}(s) \text{ (RS), }\\
% &=&m \mathbbm{1}_{[0,q^p_1)}(s)+  \mathbbm{1}_{[q^p_1,1]}(s), \text{ where } m=1 \text{ (1RSB). }
 %\end{eqnarray*}

{\section*{Acknowledgements} The author thanks Song Mei \cite{SongMei} for showing her some numerical simulations of the Parisi measure at various low temperatures.}

\section{Proof outline of Theorem \ref{mainthm}}\label{section4}

From now on, in order to simplify our notations,  we only consider the SK model with $\xi(x)=\beta^2 x^2,$ for $\beta>0$.

\subsection{Properties of Parisi Measures}

In order to prove our main results in Theorem \ref{mainthm}, 
 we now consider probability measures on $[0, 1]$ whose support contains atoms. Assume  $\mu \in M[0,1]$ has two atoms at $q_p$ and $q_{p+1}$ with $(q_p,q_{p+1})\notin$ supp{\color{white}.}$\mu_\beta$ and $\mu([0,q_p])=m_p$. We can then solve the Parisi PDE \eqref{ParisiPDE} explicitly by the Cole-Hopf transformation. To be more specific, 
%for $q_{k+1} \leq u \leq 1,$
%\begin{eqnarray*}
%\Phi_\mu(x,u)= \log \cosh x+\frac12 \big[ \xi'(1)-\xi'(u) \big]
%\end{eqnarray*}
%and 
for $q_p \leq u <q_{p+1}$,
\begin{eqnarray*}
\Phi_\mu(x,u)= \frac{1}{m_p} \log \mathbb{E} \exp m_p \Phi_\mu(x+g \sqrt{\xi'(q_{p+1})-\xi'(u)} ,q_{p+1}),
\end{eqnarray*}
where $g$ is a standard Gaussian random variable.
%For each $\mu \in M_d[0,1],$ there exists the following unique triplet $(k,\mathbf{m},\mathbf{q})$ satisfying that
%\begin{eqnarray*}
%&\mathbf{m}:& m_0=0 \leq m_1 < m_2 < \cdots <m_k \leq m_{k+1}=1, \\
%&\mathbf{q}:& q_0=0 \leq q_1 < q_2 < \cdots < q_{k+1} \leq q_{k+2}=1.
%\end{eqnarray*}
%such that $\mu([0,q_p])=m_p$ for $0 \leq p \leq k+1.$ Since the distribution of $\mu$ is a step function, 

Now in order to prove our main results, we will need a criterion to characterize the structure of the Parisi measure $\mu_{\beta}.$ Let  $B=(B(t))_{t \geq 0}$ be a standard Brownian motion and consider the time changed Brownian motion $M(u)=B(\xi'(u))$ for $u \in [0,1].$ For any $\mu \in M[0,1],$ we first define
\begin{eqnarray*}
W_\mu(u)=\int^u_0 \big( \Phi_\mu( M(u)  ,u) - \Phi_\mu(  M(s)   ,s) \big)d \mu(s), 
\end{eqnarray*}
and then
\begin{eqnarray*}
\Gamma_\mu(u)=\mathbb{E} \big( \partial_x \Phi_\mu(  M(u) ,u) \big)^2 \exp W_\mu(u),
\end{eqnarray*}
for $u \in [0,1].$
%We then consider the following two functions
%\begin{eqnarray}\label{criterionfunc1}
%F_\mu(u)= \Gamma_\mu(u)-u
%\end{eqnarray}
%and 
%\begin{eqnarray}\label{criterionfunc2}
%f_\mu(u)= \int^u_0 \frac{\xi''(s)}{2}  \cdot F_\mu(s) ds.
%\end{eqnarray}
Auffinger and Chen \cite{AChen15PTRF} proved the following necessary criterion for $\mu \in M[0,1]$ to be the Parisi measure:
\begin{theorem}[Proposition 3, Theorem 5 in \cite{AChen15PTRF}] \label{thmcriterion}
For any $\mu \in M[0,1],$ $\Gamma_\mu(u)$ is differentiable  and $\Gamma'_\mu(u)$ is continuous with respect to $u$, with
\begin{eqnarray*}
\Gamma'_\mu(u)=2\beta^2 \mathbb{E}\big[ \big( \partial^2_{x} \Phi_\mu(M(u),u) \big)^2 \exp W_\mu(u) \big].
\end{eqnarray*}
Moreover, if $\mu_\beta$ is the Parisi measure, then $\Gamma_{\mu_\beta}(u)=u$ and $\Gamma'_{\mu_\beta}(u) \leq 1$ for all $u \in \text{supp{\color{white}.}}\mu_\beta.$
\end{theorem}
%$\Gamma_\mu$ is differentiable and $\Gamma'_\mu$ is continuous with 
%\begin{eqnarray*}
%\Gamma'_\mu(u)=\xi''(u)\mathbb{E}\big[ \big( \partial^2_x\Phi_\mu( M(u) ,u) \big)^2 \exp W_\mu(u)\big].
%\end{eqnarray*}
%They also deduce the following necessary criterion for $\mu \in M[0,1]$ to be the Parisi measure of the SK model. 
%A similar criterion was obtained by Talagrand\cite{talagrand2006parisi-b} and later utilized  by Auffinger and Chen \cite{AChen15PTRF}.

We will also use the following stability fact of the Parisi measure for $\beta$ near the critical temperature $\frac{1}{\sqrt{2}}$.

\begin{lemma}\label{continuitylem}
    For every $\varepsilon > 0$, there exists $\omega>0$ such that for $|\beta - \frac{1}{\sqrt{2}}|\leq\omega$, $\mu_{\beta}([0, \varepsilon])\geq 1-\varepsilon$. %Here $\mu_{\beta}$ is the Parisi measure at finite temperature $\beta$.
\end{lemma}

This lemma follows from the beginning of Section 2 in \cite{panchenko0} and the fact that $F_{N,\beta}$ is convex in $\beta$.

We now formally define the two auxiliary functions mentioned in \S 1.5. For any probability measure $\mu$, we define the following function on $[0,1]\times[0,1],$
\begin{eqnarray*}
G_\mu(s,t)=\frac{(t-s)\cdot [\Gamma'_\mu(s)+\Gamma'_\mu(t)]}{\Gamma_\mu(t)-\Gamma_\mu(s)}-2.
\end{eqnarray*}
In particular, we define $F_\mu(x)=G_\mu(0,x).$
We then introduce some general properties of $G_\mu$ and $F_\mu$ for $\mu\in M[0,1].$ The reader should keep in mind that $\Gamma'_\mu$ and thus $G_\mu$ may not be differentiable in the support of $\mu$. 
\begin{theorem}\label{thmFderiv}
Suppose that $\frac{1}{\sqrt{2}} \leq \beta \leq 100$ and $\mu \in  M[0,1]$ where $q<q'$ are two adjacent points in supp{\color{white}.}$\mu$ with $(q,q')\notin$ supp{\color{white}.}$\mu$.
\begin{enumerate}
\item $G_\mu(q,t)$ is continuous on $(q,1].$ Moreover, $\lim_{t \rightarrow q^+} G_{\mu}(q,t)= \lim_{t \rightarrow q^+} \partial_t G_{\mu}(q,t)=0.$
%\item For any $x \in (0,q_\beta),$ $F''(x)$ exists and is continuous
\item There exists  constants $A,\eta_0>0$ independent of $\beta,q,q'$  such that $\lim_{t\rightarrow q^+}\partial_t^2 \big\{G_{\mu_\beta}(q,t) \big\}\geq A$ for $\beta \in [\frac{1}{\sqrt{2}}, \frac{1}{\sqrt{2}}+\eta_0]$ and $q \in [0,\eta_0]$.
%There exists  constants $A,\eta_0>0$ independent of $\beta,q,q'$  such that $\lim_{x \rightarrow 0^+}F''_{\delta_0}(x)\geq A$ and  $\lim_{t\rightarrow q^+}\frac{\partial^2}{\partial t^2} \big\{G_{\mu_\beta}(q,t) \big\}\geq \frac{2A}{3}$ for $\frac{1}{\sqrt{2}}\leq \beta \leq \frac{1}{\sqrt{2}}+\eta_0$ and $q \in [0,\eta_0]$.
\item For any $t \in (q,q')$, $\partial^{3}_t G_\mu(q,t)$ exists and there exists a constant $B>0$ independent of $q,q'$ and $\beta$ such that  $\big|\partial^{3}_t G_{\mu}(q,t)\big|\leq B$. 
\end{enumerate}
\end{theorem}
\begin{corollary}\label{coroFderiv}
Under the assumption of Theorem \ref{thmFderiv}, for any $\frac{1}{\sqrt{2}}\leq \beta \leq \frac{1}{\sqrt{2}}+\eta_0$ and $0 \leq q \leq \eta_0,$ we must have $G_{\mu_\beta}(q,t)>0$ for any $t \in \big(q,\min(q',q+\frac{A}{3B})\big]$.
\end{corollary}
\begin{proof}[Proof of Corollary \ref{coroFderiv}]
%Since $\lim_{\beta \rightarrow \frac{1}{\sqrt{2}}} \mu_\beta=\delta_0$ (in the sense of Lemma \ref{continuitylem}), 
By Theorem \ref{thmFderiv}, we have that  $\lim_{t \rightarrow q^+} \partial^2_{t}G_{\mu_\beta}(q,t)\geq A$, for $\frac{1}{\sqrt{2}}\leq \beta\leq\frac{1}{\sqrt{2}}+\eta_0$ and $0 \leq q \leq \eta_0.$ We then obtain that for $t \in \big(q,\min(q',q+\frac{A}{3B})\big),$
\begin{eqnarray*}
\partial^2_{t}G_{\mu_\beta}(q,t) &=& \lim_{t \rightarrow q^+} \partial^2_{t}G_{\mu_\beta}(q,t)+\int^t_q \partial^3_t G_{\mu_\beta}(q,s)ds\\
&\geq& \lim_{t \rightarrow q^+} \partial^2_{t}G_{\mu_\beta}(q,t) -\int^t_q | \partial^3_t G_{\mu_\beta}(q,s)|ds\\
&\geq& A -( t-q)B\\
&\geq& \frac{2A}{3}.
\end{eqnarray*}
%which implies that $ F''_{\mu_\beta}(x) \geq \frac{A}{3}$ for $x \in \big(0,\min(q_\beta,\frac{A}{3B})\big)$. 
We then have that $G_{\mu_\beta}(q,t) $ and $\partial_{t}G_{\mu_\beta}(q,t) $ are strictly positive for $t \in \big(q,\min(q',q+\frac{A}{3B})\big),$ with the following quantitative lower bounds:
\begin{eqnarray*}
\partial_{t}G_{\mu_\beta}(q,t)&=& \lim_{t\rightarrow q^+} \partial_{t}G_{\mu_\beta}(q,t)+\int^t_q\partial^2_{t}G_{\mu_\beta}(q,s)ds \geq (t-q) \cdot \frac{2A}{3},
\end{eqnarray*}
and
\begin{eqnarray*}
 G_{\mu_\beta}(q,t)&=& \lim_{t \rightarrow q^+} G_{\mu_\beta}(q,t)+\int^t_q \partial_t G_{\mu_\beta}(q,s)ds
\geq (t-q)^2 \cdot \frac{A}{3},
\end{eqnarray*}
for $t \in \big(q,\min(q',q+\frac{A}{3B})\big).$ Since $G_{\mu_\beta}(q,t)$ is continuous on $(q,1],$ it holds that $G_{\mu_\beta}(q,t) >0$ for $t \in \big(q,\min(q',q+\frac{A}{3B})\big].$
\end{proof}

We also consider a special case of $F_\mu$ when $\mu=\delta_0$ as follows:
\begin{theorem}\label{Fm}
$F_{\delta_0}(x)$ is strictly increasing for any $x \in [0,1]$ and $\beta>0.$  For any $\beta \geq \frac{1}{\sqrt{2}}$, there exists $M>0$ independent of $\beta$ (but may depend on $A$ and $B$) such that $F_{\delta_0}(x)\geq M,$ for $ \frac{A}{3B} \leq x \leq 1$.
\end{theorem}

Lastly,   we introduce the following lemma regarding the calculation of derivatives of $\Gamma_\mu$  in \cite{AChen15PTRF}:
\begin{theorem}[Lemma 2 in \cite{AChen15PTRF}] \label{thmcal}
For any $\mu\in M[0,1]$ be continuous on $[a,b]$ for some $a,b \in [0,1].$ Suppose that $L$ is a a polynomial on $\mathbb{R}^k.$ Define
\begin{eqnarray}\label{Fmu}
P_\mu(u)=\mathbb{E}\big[ L(\partial_x\Phi_\mu(M(u),u),\cdots, \partial^k_x \Phi_\mu(M(u),u))\exp W_\mu(u)\big]
\end{eqnarray}
for $u \in [0,1].$ Then for $u \in [a,b],$
\begin{eqnarray}\label{calgamma}
\frac{d}{du}\{ P_\mu(u) \} &=&\frac{\xi''(u)}{2} \mathbb{E} \Bigg[ \bigg( \sum^k_{i,j=1} \partial_{y_i}\partial_{y_j} L(\partial_x\Phi_\mu,\cdots, \partial^k_x \Phi_\mu )\partial^{i+1}_x \Phi_\mu \Phi^{j+1}_x \Phi_\mu \nonumber \\
&&-\mu([0,u]) \sum^k_{i=1}\sum^{i-1}_{j=1} \begin{pmatrix} i\\j\end{pmatrix} \partial_{y_i}L(\partial_x\Phi_\mu,\cdots, \partial^k_x \Phi_\mu ) \partial^{j+1}_x \Phi_\mu \partial^{i-j+1}_x \Phi_\mu \bigg) \exp W_\mu(u)\Bigg].
\end{eqnarray}

\end{theorem}

\subsection{Proof Outline of Theorem \ref{mainthm}}

In this section, we {give a more detailed outline of} the proof for our main results. 
In order to prove our main results, we will first assume $\beta \in \big(\frac{1}{\sqrt{2}}, \frac{1}{\sqrt{2}}+\eta\big]$ for some $\eta>0$ and then prove  supp{\color{white}.}$\mu_\beta=[0,\upsilon_\beta] $, for some $\upsilon_\beta\in(0,\eta)$ by ruling out the following three cases. {\color{black}Below $\mathcal{K}>0$ will be a small absolute constant chosen in the discussion of Case I.}

\subsubsection{Case I: {\color{black}$(q,q') \cap$ supp{\color{white} .}$\mu_{\beta}=\emptyset$,   $q\leq\mathcal{K}$, $q'-q>\frac{A}{3B}$ and $q,q' \in$ supp{\color{white} .}$\mu_{\beta}$.}}
%$\mathcal{K} \notin$ supp{\color{white} .}$\mu_{\beta}$ and supp{\color{white} .}$\mu_{\beta}$} is not contained in $[0, \mathcal{K}]$.}

%{\color{black}In this case, because $0\in$ supp{\color{white}.}$\mu_{\beta}$ \cite[Theorem 1]{AChen15PTRF}, we can assume}  $q<q'$ are two points in supp{\color{white}.}$\mu_\beta$ satisfying that  $(q,q') \notin$ supp{\color{white}.}$\mu_\beta$ {\color{black}and $q< \mathcal{K}< q'$}.  
As mentioned in \S\ref{proofidea}, we will use $G_\mu$  to prove the four relations  that $q,q'$ needs to satisfy in Theorem \ref{thmcriterion}
\begin{eqnarray}\label{4relations}
\Gamma_{\mu_\beta}(q)=q, \Gamma_{\mu_\beta}(q')=q',\Gamma'_{\mu_\beta}(q)\leq1 \text{ and } \Gamma'_{\mu_\beta}(q')\leq 1
\end{eqnarray}
cannot hold simultaneously.
%In order to rule out this case, we construct a quantity ($F_\mu$ below) to prove the above three relations cannot hold simultaneously.  that $\mu_\beta$ cannot be the Parisi measure for $\beta$, we show that the two conditions $\Gamma_{\mu_\beta}'(0)\leq1$ and $\Gamma_{\mu_\beta}'(q_\beta)\leq1$ can not hold simultaneously. 

%Next we introduce the construction of the crucial $F_\mu$ and use it to obtain a contradiction in this case. For any probability measure $\mu$, we consider the following function on $[0,1]$
%\begin{eqnarray*}
%G_\mu(s,t)=\frac{(t-s)\cdot [\Gamma'_\mu(s)+\Gamma'_\mu(t)]}{\Gamma_\mu(t)-\Gamma_\mu(s)}-2.
%\end{eqnarray*}
%In particular, we define $F_\mu(x)=G_\mu(0,x).$

%\begin{eqnarray*}
%G_\mu()=\frac{x\cdot [\Gamma'_\mu(0)+\Gamma'_\mu(x)]}{\Gamma_\mu(x)}-2.
%\end{eqnarray*}

%As we will prove below, $F_\mu$ is continuous on $[0,1].$ 

We will show that there exists {\color{black}choices} $\eta, \mathcal{K}>0$, such that $G_{\mu_\beta}(q,q')>0$ for $\frac{1}{\sqrt{2}}<\beta\leq\frac{1}{\sqrt{2}}+\eta$ and $0\leq q \leq \mathcal{K}$. But this contradicts at least one of the four above properties: $\Gamma_{\mu_\beta}(q)=q, \Gamma_{\mu_\beta}(q')=q',\Gamma'_{\mu_\beta}(q)\leq1 \text{ and } \Gamma'_{\mu_\beta}(q')\leq 1$. 
%We will derive this contradiction from Corollary \ref{coroFderiv} and the theorem below.

%We will show that there exists $\varepsilon>0$, such that for $\frac{1}{\sqrt{2}}<\beta<\frac{1}{\sqrt{2}}+\varepsilon,$ it holds that $F_{\mu_\beta}(q_\beta)>0$. Therefore for $q_\beta$ satisfying that $\Gamma_{\mu_\beta}(q_\beta)=q_\beta$, the two inequalities $\Gamma'_{\mu_\beta}(0)\leq1$ and $\Gamma'_{\mu_\beta}(q_\beta)\leq1$ can not hold simultaneously.

%Since $\lim_{\beta \rightarrow \frac{1}{\sqrt{2}}} \mu_\beta=\delta_0$ {(in the sense of Lemma \ref{continuitylem})} and $\mu_\beta \neq \delta_0$ for $\beta>\frac{1}{\sqrt{2}},$ there exists $0<\eta_1<\min(\eta_0,0.1)$, such that whenever $\frac{1}{\sqrt{2}}<\beta\leq \frac{1}{\sqrt{2}}+\eta_1$, {we have \color{black}$1-\varepsilon<m \leq 1$}% or $0\leq q< q'<\varepsilon$ holds 
%(where $\varepsilon \in \big(0,\min(\frac{A}{3B},\eta_1)\big)$ is small and to be determined). We now define $\mathcal{K}=\min\big( \frac{A}{3B},\eta_1 \big)$ and {\color{black}obtain $G_{\mu_\beta}(q,t) >0$ for $t \in (q,q+\frac{A}{3B}]$ by ... (add) This is a contradiction. Thus we have ruled out Case I}

With the above preparation, now we are ready to rule out Case I. {\color{black}Suppose that $\mu_\beta([0,q])=m$.  We first note that $\lim_{\beta \rightarrow \frac{1}{\sqrt{2}}} \mu_\beta=\delta_0$ {(in the sense of Lemma \ref{continuitylem})}. Thus when $\beta$ is very close to $\frac{1}{\sqrt{2}}$, $m$ must be very close to $1$. Now by Theorem \ref{Fm} and the continuity of $G_\mu(q,\cdot)$ in $\mu$ and $q$, there exists $0<\eta_1<\min(\eta_0,0.1)$ such that whenever $\frac{1}{\sqrt{2}}<\beta\leq\frac{1}{\sqrt{2}}+\eta_1$, we  have that  for $0 \leq q \leq \eta_1$,
\begin{eqnarray*}
    G_{\mu_\beta}(q,t)>\frac{M}{2}>0, \text{ for }q+ \frac{A}{3B} \leq t \leq 1,
\end{eqnarray*} 
and for $0 \leq q \leq 1,$
\begin{eqnarray}\label{Gammaapprox}
\big|\Gamma'_{\mu_\beta}(q)- \Gamma'_{\delta_0}(q)\big|< 0.05 \text{ and } \big|\Gamma''_{\mu_\beta}(q^-)- \Gamma''_{\delta_0}(q)\big|< 0.05.
\end{eqnarray}
\eqref{Gammaapprox} will not be used here but in Case III. We state it here so that we can now make the choice of $\mathcal{K}$.

From this point on, we will fix $\mathcal{K}=\eta_1$. Hence in this case  $G_{\mu_\beta}(q,q')>0$, which contradicts \eqref{4relations}. Thus we have ruled out Case I.}

\subsubsection{Case II: {\color{black}$(q,q') \cap$ supp{\color{white} .}$\mu_{\beta}=\emptyset$,   $q\leq\mathcal{K}$, $q'-q\leq \frac{A}{3B}$ and $q,q' \in$ supp{\color{white} .}$\mu_{\beta}$.}} 

%Suppose that   $q\in \big[0,  \mathcal{K} \big)$ is an accumulation point in supp{\color{white}.}$\mu_\beta$  but $[q,\upsilon]\not\subseteq$ supp{\color{white}.}$\mu_\beta$ for every $\upsilon\in(q,1]$. 

%Suppose $q_1,q_2\in(q,1]$ are two points in  supp{\color{white}.}$\mu_\beta$ with $(q_1,q_2) \notin$ supp{\color{white}.}$\mu_\beta$. 
Recall from \eqref{4relations} that if $\mu_\beta$ is the Parisi measure for  $\beta$, then it holds that \begin{eqnarray*}
\Gamma_{\mu_\beta}(q)=q,\Gamma_{\mu_\beta}(q')=q',\Gamma_{\mu_\beta}'(q)\leq1 \text{ and }\Gamma_{\mu_\beta}'(q')\leq1.
\end{eqnarray*}
%As mentioned in \S\ref{proofidea}, %in order to prove that $\mu_\beta$ cannot be the Parisi measure for $\beta$, 
%we construct an auxiliary function ($G_\mu$ below) to show that the four relations \eqref{4relations} with respect to $q_1$ and $q_2$ 
 %cannot hold simultaneously. To be more specific, for any probability measure $\mu$, we consider the following generalized function of $F_\mu$ on $[0,1] \times [0,1]$
%\begin{eqnarray*}
%G_\mu(s,t)=\frac{(t-s)\cdot [\Gamma'_\mu(s)+\Gamma'_\mu(t)]}{\Gamma_\mu(t)-\Gamma_\mu(s)}-2.
%\end{eqnarray*}
%Note that $G_\mu(0,t)=F_\mu(t).$

%By Theorem \ref{thmFderiv}, there exist  constants $C>0$ and $\eta_2>0$ independent of $q_1,q_2$ and $\beta$   such that if $q_1 \in [0,\eta_0]$, then $\lim_{t\rightarrow q^+_1}\frac{\partial^2}{\partial t^2} \big\{G_{\mu_\beta}(q_1,t) \big\}\geq C$ for $\frac{1}{\sqrt{2}}\leq \beta \leq \frac{1}{\sqrt{2}}+\eta_0$.

%for any $\frac{1}{\sqrt{2}}\leq \beta \leq \frac{1}{\sqrt{2}}+\eta_0$ and $0 \leq q < \min\big( \frac{A}{3B},\eta_0 \big),$ we must have $G_{\mu_\beta}(q,t)>0$ for any $t \in \big(q,\min(q',q+\frac{A}{3B})\big]$.

By Corollary \ref{coroFderiv}, %for $\eta_2 \in (0,\eta_0)$, 
since $0 \leq q \leq \mathcal{K}<\eta_0$ and $q'-q\leq \frac{A}{3B}$,  we must have $G_{\mu_\beta}(q,q')>0$ for any  $\frac{1}{\sqrt{2}}\leq \beta \leq \frac{1}{\sqrt{2}}+\eta_1$, which contradicts \eqref{4relations}. We then have also ruled out Case II for $\beta\in [\frac{1}{\sqrt{2}}, \frac{1}{\sqrt{2}}+\eta_1] $.
%Note that  $C,D$ and $\eta_3$ are strictly positive constants independent of the choice of $q_1$ and $q_2$. 
%Since $q$ is an accumulation point in $[0, \mathcal{K} )$ of supp{\color{white}.}$\mu_\beta$ and  $[q,\upsilon] \not\subseteq$ supp{\color{white}.}$\mu_\beta$, for every $\upsilon\in(q,1]$, we can then  always choose a pair of adjacent points $q_1<q_2$ in supp{\color{white}.}$\mu_\beta$ satisfying that $(q_1,q_2) \notin$ supp{\color{white}.}$\mu_\beta$  and $0 \leq q_1<q_2 \leq \mathcal{K}$. Therefore by Corollary \ref{coroFderiv}, $G_{\mu_\beta}(q_1,q_2)>0,$ 

%for any $\frac{1}{\sqrt{2}}\leq \beta \leq \frac{1}{\sqrt{2}}+\eta_0$ and $0 \leq q_1 \leq \min\big( \frac{A}{3B},\eta_0 \big),$ we must have $G_{\mu_\beta}(q_1,t)>0$ for any $t \in \big(q_1,\min(q_2,\frac{A}{3B})\big]$.

\subsubsection{Case III: $[\mathcal{K},1]\cap$ supp{\color{white} .}$\mu_{\beta}\not=\emptyset$.}\label{CaseIII}

{\color{black}We rule out this case by the method of the proof of Theorem 4 in \cite{AChen15PTRF}.}
 
For $\beta \in[ \frac{1}{\sqrt{2}},\frac{1}{\sqrt{2}}+\eta_1], $ we define 
$q_0:=\sup\{ q | [0,q] \subseteq \text{supp\color{white}.}\mu_\beta \text{ for some }0 <q\leq 1 \}$. 
Without loss of generality, we may assume that $q_0\geq \mathcal{K}.$ Indeed, if $0 \leq q_0 < \mathcal{K} ,$ combining with the assumption that $[\mathcal{K},1]\cap$ supp{\color{white} .}$\mu_{\beta}\not=\emptyset$, it will  then be ruled out by the previous two cases. We define $m_1:= \mu_\beta([0,q_0))$ and $m_2:= \mu_\beta([0,q_0])$. %and $q_M:= \max \{ q \in [0,1] | q \in \text{supp{\color{white}.}}\mu_\beta \}.$

%Since $\mathcal{K}\leq 0.1,$ 
Recall from \eqref{Gammaapprox}, for $ \frac{1}{\sqrt{2}}< \beta \leq \frac{1}{\sqrt{2}}+\eta_1$, we have that
%$$\mu_{\beta}([0, \mathcal{K}])>1- \mathcal{K}\geq0.9,$$
\begin{eqnarray*}
%&\big|(\partial^2_x \Phi_{\mu_\beta}(x,q_0))^2- ( \cosh x)^{-4}\big|< 0.1&\\
%&\big|(\partial^3_x \Phi_{\mu_\beta}(x,q_0))^2- 4(\cosh^{-4} x-\cosh^{-6} x)\big|< 0.1& \\
&\big|\Gamma'_{\mu_\beta}(q_0)- \Gamma'_{\delta_0}(q_0)\big|< 0.05 \text{ and } \big|\Gamma''_{\mu_\beta}(q_0^-)- \Gamma''_{\delta_0}(q_0)\big|< 0.05.&
\end{eqnarray*}
Since we exclude Cases I and II, based on the definition of $q_0$ and Theorem \ref{thmcriterion}, it then holds that $\Gamma_{\mu_\beta}'(q_0)=1$ and $\Gamma_{\mu_\beta}''(q_0^-)=0.$   
Observe that $\Phi_{\delta_0}(x,q_0)=\log \cosh x+\beta^2(1-q_0)$ and then  by Theorem \ref{thmcriterion},
\begin{eqnarray*}
\Gamma'_{\delta_0}(q_0)&=&2 \beta^2 \mathbb{E}\big[ \cosh(M(q_0))^{-4} \exp W_{\delta_0}(q_0) \big].
%\Gamma''_{\delta_0}(q_0)&=&2\beta^2 \mathbb{E}\big[ 4\big(\cosh^{-4}(M(q_0))-6\cosh^{-6}(M(q_0))\big) \exp W_{\mu_\beta}(q_0)  \big].
\end{eqnarray*}
Also by Lemma \ref{thmcal}, we compute $\Gamma''_{\delta_0}(q_0)$ as follows:
\begin{eqnarray*}
%\Gamma'_{\delta_0}(q_0)&=&2 \beta^2 \mathbb{E}\big[ \cosh(M(q_0))^{-4} \exp W_{\mu_\beta}(q_0) \big].
\Gamma''_{\delta_0}(q_0)&=&2\beta^2 \mathbb{E}\big[ 4\big(\cosh^{-4}(M(q_0))-6\cosh^{-6}(M(q_0))\big) \exp W_{\delta_0}(q_0)  \big].
\end{eqnarray*}
We then obtain that
\begin{eqnarray*}
\big| 2 \beta^2 \mathbb{E}\big[ \cosh(M(q_0))^{-4} \exp W_{\mu_\beta}(q_0) \big]-1\big| <0.05,
\end{eqnarray*}
and 
\begin{eqnarray*}
2\beta^2 \Big| \mathbb{E}\big[ 4\big(\cosh^{-4}(M(q_0))-6\cosh^{-6}(M(q_0))\big) \exp W_{\mu_\beta}(q_0)  \big] \Big| <0.05.
\end{eqnarray*}
%Here we use the relation that $\mathbb{E}[ \exp W_{\mu_\beta}(q_0) ]=1$ and 
%\begin{eqnarray*}
%\Gamma'_{\delta_0}(q_0)&=&2 \beta^2 \mathbb{E}\big[ \cosh(M(q_0))^{-4} \exp W_{\mu_\beta}(q_0) \big], \\
%\Gamma''_{\delta_0}(q_0)&=&2\beta^2 \mathbb{E}\big[ 4\big(\cosh^{-4}(M(q_0))-6\cosh^{-6}(M(q_0))\big) \exp W_{\mu_\beta}(q_0)  \big].
%\end{eqnarray*}
Combining the inequalities above, by Jensen's inequality, we obtain that
\begin{eqnarray}\label{contradiction3}
0.95 &\leq&   2 \beta^2 \mathbb{E}\big[ \cosh(M(q_0))^{-4} \exp W_{\mu_\beta}(q_0) \big] \nonumber \\
&\leq & 2 \beta^2 \Big( \mathbb{E}\big[ \cosh(M(q_0))^{-6} \exp W_{\mu_\beta}(q_0) \big] \Big)^{2/3} \nonumber \\
&\leq & 2 \beta^2 \Big( \frac{1}{240\beta^2}+  \frac{2}{3} \mathbb{E}\big[ \cosh(M(q_0))^{-4} \exp W_{\mu_\beta}(q_0) \big]  \Big)^{2/3}  \nonumber \\
& \leq & 2 \beta^2 \Big( \frac{1}{240\beta^2}+  \frac{2}{3}  \cdot \frac{1.05}{2\beta^2}  \Big)^{2/3} \nonumber\\
&=& 2 \Big( \frac{1}{240}+  \frac{2}{3}  \cdot \frac{1.05}{2}  \Big)^{2/3} {\beta^{2/3} } .
\end{eqnarray}
Here  we use the relation that $\mathbb{E}[\exp W_{\mu_\beta}(q_0)]=1$ in the second inequality.
However, since we only consider $\beta$ slightly beyond the high temperature regime, i.e. $\beta \in [\frac{1}{\sqrt{2}}, \frac{1}{\sqrt{2}}+\eta_1]$ and $\eta_1<0.1$, we then obtain that $\beta<0.85$ and
$$2 \Big( \frac{1}{240}+  \frac{2}{3}  \cdot \frac{1.05}{2}  \Big)^{2/3} {\beta^{2/3} } < 0.9,$$
which leads to a contradiction with \eqref{contradiction3}. We then have also ruled out Case III for $\beta\in [\frac{1}{\sqrt{2}}, \frac{1}{\sqrt{2}}+\eta_1] $.

Therefore, for any $\beta\in (\frac{1}{\sqrt{2}}, \frac{1}{\sqrt{2}}+\eta_1] $, all the three cases discussed above have been ruled out for $\mu_\beta$. Since $\mu_\beta \not= \delta_0$ for $\beta\in (\frac{1}{\sqrt{2}}, \frac{1}{\sqrt{2}}+\eta_1] $, then the only possibility is that supp{\color{white}.}$\mu_\beta=[0,v_\beta]$ for some $v_\beta\in (0,\eta_1).$ {\color{black}Hence we have proved Theorem \ref{mainthm}.

\subsection{Proof of Corollary \ref{coromain}} Next we will derive Corollary \ref{coromain}  directly from  Theorem \ref{mainthm} by a result of Auffinger-Chen\cite{AChen15PTRF}.}

\begin{proof}[Proof of Corollary \ref{coromain}]
By Theorem 2 in \cite{AChen15PTRF}, since $(0,\upsilon_\beta)\subseteq$ supp{\color{white}.}$\mu_\beta$, then the distribution function of $\mu_\beta$ is infinitely differentiable on $[0,\upsilon_\beta)$. {\color{black} Also, by   taking the limit $u \to 0^+$ in (21) of \cite{AChen15PTRF} we obtain $\mu_{\beta} (\{0\}) = 0$.}

Moreover by Theorem 2 in \cite{AChen15PTRF}, for $\beta\in(\frac{1}{\sqrt{2}},\frac{1}{\sqrt{2}}+\eta_1],$ $\mu_\beta$ has a jump discontinuity at $\upsilon_\beta,$ which guarantees that $m<1$.

\end{proof}

\subsection{Proof of Theorems \ref{thmFderiv} and \ref{Fm}}
In this section, we prove the three theorems stating the crucial properties of $F_\mu$ and $G_\mu.$ 
We first prove Theorem \ref{Fm} as follows:
\begin{proof}[Proof of Theorem \ref{Fm}]
Note that when $\mu=\delta_0,$ we have that  for $u \in [0,1]$, 
\begin{eqnarray*}
W_\mu(u)&=&\Phi\big( M(u) ,u\big)-\Phi(0,0)\\
&=& \log\cosh  M(u)  +\frac{1}{2}\big( \xi'(1)-\xi'(u) \big)- \frac{1}{2} \xi'(1)\\
&=& \log\cosh  M(u)  -\frac{1}{2}\xi'(u).
\end{eqnarray*}
We also have that
\begin{eqnarray*}
\Phi_\mu(x,u) &=& \log \cosh(x)+\frac{1}{2} [\xi'(1)-\xi'(u)], \\
\partial_x \Phi_\mu(x,u)
%&=&\frac{\mathbb{E}[\cosh(x+g \sqrt{\xi'(1)-\xi'(u)} ) \tanh(x+g \sqrt{\xi'(1)-\xi'(u)} )]}{\mathbb{E}[ \cosh(x+g \sqrt{\xi'(1)-\xi'(u)} )]}\\
%&=&\frac{\sinh(x)\cdot \exp\big(\frac12(\xi'(1)-\xi'(u))\big)}{\cosh(x)\cdot \exp\big(\frac12(\xi'(1)-\xi'(u))\big)}\\
&=&\tanh(x),\\
%\end{eqnarray*}
%and
%\begin{eqnarray*}
\partial^2_x \Phi_\mu(x,u)&=& \cosh^{-2}(x).
\end{eqnarray*}
Therefore we obtain that for $u \in [0,1],$
\begin{eqnarray*}
\Gamma_{\delta_0}(x)&=&  \exp( -\beta^2 x ) \mathbb{E}[\tanh^2\big( M(x)  \cosh  M(x) \big)] \\
&=&\frac{\mathbb{E}[\tanh^2(\sqrt{2\beta^2x}g)\cosh(\sqrt{2\beta^2x}g)]}{\mathbb{E}[\cosh(\sqrt{2\beta^2x}g)]},
\end{eqnarray*}
and
\begin{eqnarray*}
\Gamma'_{\delta_0}(x)&=&2\beta^2\frac{\mathbb{E}\big[ \cosh^{-3}(\sqrt{2\beta^2x}g) \big]}{\mathbb{E}[ \cosh(\sqrt{2\beta^2x}g) ]},\\
\Gamma'_{\delta_0}(0)&=&2\beta^2,
\end{eqnarray*}
where $g$ is a standard Gaussian random variable. Here we use the relation $ \exp( \beta^2 x )= \mathbb{E}[ \cosh(\sqrt{2\beta^2x}g) ].$

Based on the ingredients above, we can then prove the strict monotonicity of $F_{\delta_0}$ by computing its derivative as follows:
\begin{eqnarray*}
\frac{d}{dx} \big\{ F_{\delta_0}(x) \big\}&=&\frac{1}{\big(\Gamma_{\delta_0}(x)\big)} \cdot \Big[ \Gamma'_{\delta_0}(0)+\Gamma'_{\delta_0}(x)+ x \cdot \frac{d}{dx}\big( \Gamma'_{\delta_0}(x) \big) \Big] \\
&&- \frac{x \cdot \big( \Gamma'_{\delta_0}(0)+\Gamma'_{\delta_0}(x) \big)}{\big(\Gamma_{\delta_0}(x)\big)^2} \cdot \frac{d}{dx} \big( \Gamma_{\delta_0}(x)\big) \\
&=& \frac{a_1(x)}{a_1(x)-a_{-1}(x)} \cdot \Big[ 2\beta^2\Big(1+\frac{a_{-3}(x)}{a_1(x)}\Big) +2\beta^4 x \cdot \frac{\big(8a_{-3}(x)-12a_{-5}(x)\big)}{a_1(x)} \Big] \\
&&+ \frac{2\beta^2x (a_1(x))^2\Big(1+\frac{a_{-3}(x)}{a_1(x)}\Big)}{\big( a_1(x)-a_{-1}(x) \big)^2} \cdot \frac{2\beta^2a_{-3}(x)}{a_1(x)} \\
&=& \frac{2 \beta^2}{\big(a_1(x)-a_{-1}(x)\big)^2} \Big\{ \big(a_1(x)+a_{-3}(x)\big) \big(a_1(x)-a_{-1}(x)\big)\\
&&-2 \beta^2x\Big[ \big(a_1(x)-a_{-1}(x)\big)\big(6a_{-5}(x)-4a_{-3}(x)\big)+ a_{-3}(x)\big( a_1(x)+a_{-3}(x) \big)\Big] \Big\} \\
&:=&  \frac{2 \beta^2}{\big(a_1(x)-a_{-1}(x)\big)^2} \cdot  f_{\delta_0}(x),
\end{eqnarray*}
where $a_n(x):= \mathbb{E}[\cosh^n(\sqrt{2\beta^2x}g)]$, for $n \in \mathbb{Z}.$

In order to prove that $F_{\delta_0}(x)$ is strictly increasing on $(0,1],$ it suffices for us to show that $f_{\delta_0}(x)>0$ for $x \in (0,1].$ We now split the terms in $f_{\delta_0}(x)$ into the following two groups:
\begin{eqnarray*}
I:=\big(a_1(x)+a_{-3}(x)\big) \big(a_1(x)-a_{-1}(x)\big) - 2\beta^2x a_{-3}(x) \cdot \big( a_1(x)+a_{-3}(x) \big),
\end{eqnarray*}
and
\begin{eqnarray*}
II:=-2 \beta^2x \cdot \big(a_1(x)-a_{-1}(x)\big)\big(6a_{-5}(x)-4a_{-3}(x)\big).
\end{eqnarray*}
We then prove that 
\begin{eqnarray} \label{group1}
I- \frac{2\beta^2x}{1+6\beta^2x} \big(a_1(x)+a_{-3}(x)\big) \big(a_1(x)-a_{-1}(x)\big) >0,
\end{eqnarray}
and
\begin{eqnarray}\label{group2}
II+ \frac{2\beta^2x}{1+6\beta^2x} \big(a_1(x)+a_{-3}(x)\big) \big(a_1(x)-a_{-1}(x)\big) >0,
\end{eqnarray}
for $x \in (0,1].$
For $x \in (0,1],$ the inequality \eqref{group1} is equivalent to 
\begin{eqnarray*}
a_1(x)-a_{-1}(x)-2\beta^2xa_{-1}(x)+2\beta^2x \cdot \big( 2a_1(x)-a_{-1}(x)-a_{-3}(x)-6\beta^2xa_{-3}(x) \big)>0,
%\\(1+4\beta^2x)\big(a_1(x)-a_{-1}(x) \big) -2\beta^2x (1+6\beta^2x)a_{-3}(x) >0,
\end{eqnarray*}
and \eqref{group2} is equivalent to
\begin{eqnarray*}
a_1(x)+5a_{-3}(x)-6a_{-5}(x)-2\beta^2x \big( 18a_{-5}(x)-12a_{-3}(x)\big)>0.
\end{eqnarray*}
Note that by an application of Gaussian integration by parts, we have the relation 
\begin{eqnarray*}
\mathbb{E}\Big[\big(\sqrt{2\beta^2x}g\big) b_n\big( \sqrt{2\beta^2x}g \big)\Big]=2\beta^2x \cdot a_n(x),
\end{eqnarray*}
where $b_n(x)$ is the antiderivative of $\cosh^n(x)$.
We then obtain the following three relations:
\begin{eqnarray*}
&(1)&a_1(x)-a_{-1}(x)-2\beta^2xa_{-1}(x) \\
&&= \mathbb{E}\Big[ \cosh\big( \sqrt{2\beta^2x}g \big)- \cosh^{-1}\big( \sqrt{2\beta^2x}g \big) - \big(\sqrt{2\beta^2x}g\big) \cdot   b_{-1}\big( \sqrt{2\beta^2x}g \big)\Big],\\
&(2)&2a_1(x)-a_{-1}(x)-a_{-3}(x)-6\beta^2xa_{-3}(x)\\
&&= \mathbb{E}\Big[ 2\cosh\big( \sqrt{2\beta^2x}g \big)- \cosh^{-1}\big( \sqrt{2\beta^2x}g \big)- \cosh^{-3}\big( \sqrt{2\beta^2x}g \big) -3 \big(\sqrt{2\beta^2x}g\big) \cdot   b_{-3}\big( \sqrt{2\beta^2x}g \big)\Big]\\
&(3)&a_1(x)+5a_{-3}(x)-6a_{-5}(x)-2\beta^2x \big( 18a_{-5}(x)-12a_{-3}(x)\big),\\
&&= \mathbb{E}\Big[ \cosh\big( \sqrt{2\beta^2x}g \big) +5 \cosh^{-3}\big( \sqrt{2\beta^2x}g \big)-6 \cosh^{-5}\big( \sqrt{2\beta^2x}g \big)\\
&&- \big(\sqrt{2\beta^2x}g\big) \cdot   \Big(18 b_{-5}\big( \sqrt{2\beta^2x}g \big)-12 b_{-3}\big( \sqrt{2\beta^2x}g \big) \Big)\Big],
\end{eqnarray*}

Now in order to prove that $F_{\delta_0}(x)$ is strictly increasing on $(0,1],$ it suffices for us to show that the following three inequalities holds for any $x \in \mathbb{R} \setminus \{0\} $:
\begin{eqnarray}\label{GIP123}
&(1)& \cosh( x )- \cosh^{-1}( x ) - x \cdot   b_{-1}( x ) > 0, \label{GIP1} \\
&(2)& 2 \cosh(x)-\cosh^{-1}(x)-\cosh^{-3}(x)-3x \cdot b_{-3}(x) >0 , \label{GIP2}\\
&(3)& \cosh(x)+5\cosh^{-3}(x)-6\cosh^{-5}(x)-6x \big(3b_{-5}(x)-2b_{-3}(x) \big) >0. \label{GIP3}
\end{eqnarray}
 We leave the proof of the three inequalities to the end of the section.
\end{proof}

Now we  prove Theorem \ref{thmFderiv}:
\begin{proof}[Proof of Theorem \ref{thmFderiv}]
Recall that 
\begin{eqnarray*}
G_\mu(s,t)=\frac{(t-s)\cdot [\Gamma'_\mu(s)+\Gamma'_\mu(t)]}{\Gamma_\mu(t)-\Gamma_\mu(s)}-2.
\end{eqnarray*}
and 
\begin{eqnarray*}
\Gamma'_\mu(u)=2 \beta^2 \cdot \mathbb{E}\big[ \big( \partial^2_{x} \Phi_\mu(M(u),u) \big)^2 \exp W_\mu(u) \big].
\end{eqnarray*}
By the definition of $\Gamma'_\mu(u),$ we have that $\Gamma'_\mu(u)>0$ for $u \in (0,1],$ which implies that $\Gamma_{\mu}(u) \neq \Gamma_{\mu}(q)$ for $u \in [0,1] \setminus \{q\}$. Since $\Gamma_\mu'(u)$ is continuous, for $t \in (q,1]$, we then have that $G_\mu(q,t)$ is continuous on $(q,1].$

By Proposition 1 in \cite{AChen15PTRF}, any function of the form \eqref{Fmu} is a continuous function and uniformly on $[0,1].$
%\begin{eqnarray*}
%\big|\mathbb{E}[\partial^j_x \Phi_\mu (M(u),u) \exp W_\mu(u)] \big| \leq  C_{0,j},
%\end{eqnarray*}
%for $u \in [0,1]$ and $j \geq 1.$ 
Then by Theorem \ref{thmcal}, $\Gamma^{(k)}_\mu(u)$ is continuous with respect to $u$ on $(q,q')$ for $k \geq 0$. Now for any $\beta>0$, we compute $\lim_{t \rightarrow q^+}G_\mu(q,t)$ and $\lim_{t\rightarrow q^+}\frac{\partial}{\partial t}\{G_\mu(q,t)\}$ by L'H$\hat{o}$pital's rule as follows:
\begin{eqnarray*}
\lim_{t \rightarrow q^+}G_\mu(q,t)&=& \lim_{t \rightarrow q^+} \frac{  [\Gamma'_\mu(q)+\Gamma'_\mu(t)] +(t-q)\cdot \Gamma''_\mu(t) }{ \Gamma'_\mu(t) }-2 \\
&=& \frac{2\Gamma'_\mu(q)}{\Gamma'_\mu(q)}-2\\
&=&0,
\end{eqnarray*}
and
\begin{eqnarray*}
&&\lim_{t\rightarrow q^+}\frac{\partial}{\partial t}\{G_\mu(q,t)\}\\
&=& \lim_{t\rightarrow q^+} \frac{1}{2\Gamma'_\mu(t)[\Gamma_\mu(t)-\Gamma_\mu(q)]} \cdot \big\{ [\Gamma_\mu(t)-\Gamma_\mu(q)]\cdot [2\Gamma''_\mu(t)+\Gamma'''_\mu(t)(t-q)]\\
&&-\Gamma_\mu(t)[\Gamma'_\mu(t)+\Gamma'_\mu(q)](t-q) \big\} \\
&=& \lim_{t\rightarrow q^+} \frac{1}{2\Gamma'_\mu(t)^2+2\Gamma''_\mu(t)[\Gamma_\mu(t)-\Gamma_\mu(q)]} \cdot \Big\{ \Gamma'_\mu(t)[2\Gamma''_\mu(t)+\Gamma'''_\mu(t)(t-q)] \\
&& +[\Gamma_\mu(t)-\Gamma_\mu(q)][3\Gamma'''_\mu(t)+\Gamma^{(4)}_\mu(t)(t-q)]- \Gamma'''_\mu(t)[\Gamma'_\mu(t)+\Gamma'_\mu(q)](t-q)\\
&&-\Gamma''_\mu(t)[\Gamma'_\mu(t)+\Gamma'_\mu(q)]-\Gamma'_\mu(t)^2(t-q)\Big\} \\
&=& \lim_{t\rightarrow q^+}  \frac{1}{2\Gamma'_\mu(t)^2}\big\{2 \Gamma'_\mu(t) \Gamma''_\mu(t)-\Gamma''_\mu(t) \cdot [\Gamma'_\mu(t)+\Gamma'_\mu(q)] \big\}\\
&=&0.
\end{eqnarray*}
We then consider $\lim_{t\rightarrow q^+}\frac{\partial^2}{\partial t^2} \big\{G_{\mu_\beta}(q,t) \big\},$ for $\beta \geq \frac{1}{\sqrt{2}}.$ We compute $\frac{\partial^2}{\partial t^2} \big\{G_{\mu_\beta}(q,t) \big\}$ as follows:
\begin{eqnarray*}
&&\frac{\partial^2}{\partial t^2} \big\{ G_\mu(q,t) \big\} \\
&=&\frac{1}{[\Gamma_\mu(t)-\Gamma_\mu(q)]^3} \Big\{ [\Gamma_\mu(t)-\Gamma_\mu(q)]\big[ [\Gamma_\mu(t)-\Gamma_\mu(q)][2 \Gamma''_\mu(t)+\Gamma'''_\mu(t)(t-q)]\\
&&- \Gamma''_\mu(t)[\Gamma'_\mu(q)+3 \Gamma'_\mu(t)](t-q)\big] -2 \Gamma'_\mu(t) [\Gamma'_\mu(t)+\Gamma'_\mu(q)] [\Gamma_\mu(t)-\Gamma_\mu(q)-\Gamma'_\mu(t)(t-q)] \Big\}\\
&:=& \frac{1}{[\Gamma_\mu(t)-\Gamma_\mu(q)]^3} \cdot A.
\end{eqnarray*}
In order to  compute $\lim_{t\rightarrow q^+}\frac{\partial^2}{\partial t^2} \big\{G_{\mu_\beta}(q_1,t) \big\}$ by L'H$\hat{o}$pital's rule, we note that 
%the first and second derivative of $[\Gamma_\mu(t)-\Gamma_\mu(q_1)]^3$ and $A$ are zero when $t=q_1:$
\begin{eqnarray}\label{q120}
\lim_{t \rightarrow q} [\Gamma_\mu(t)-\Gamma_\mu(q)]^3=\lim_{t \rightarrow q} \frac{\partial}{\partial t} \big\{ [\Gamma_\mu(t)-\Gamma_\mu(q)]^3\big\}=\lim_{t \rightarrow q^+} \frac{\partial^2}{\partial t^2} \big\{ [\Gamma_\mu(t)-\Gamma_\mu(q)]^3\big\}=0,
\end{eqnarray}
and
\begin{eqnarray}\label{A0}
\lim_{t \rightarrow q^+}A=\lim_{t \rightarrow q^+} \frac{\partial}{\partial t} \big\{A\big\}=\lim_{t \rightarrow q^+} \frac{\partial^2}{\partial t^2} \big\{ A\big\}=0.
\end{eqnarray}
Also, the third derivative of $[\Gamma_\mu(t)-\Gamma_\mu(q_1)]^3$ and $A$ are nonzero when $t\rightarrow q^+:$
\begin{eqnarray*}
\lim_{t \rightarrow q^+} \frac{\partial^3}{\partial t^3} \big\{ A\big\}=2 \Gamma'''_\mu(q^+) \Gamma_\mu'(q)^2,
\end{eqnarray*}
and
\begin{eqnarray*}
\lim_{t \rightarrow q^+} \frac{\partial^3}{\partial t^3} \big\{ [\Gamma_\mu(t)-\Gamma_\mu(q)]^3\big\}=6\Gamma_\mu'(q)^3.
\end{eqnarray*}
We then compute $\lim_{t\rightarrow q^+}\frac{\partial^2}{\partial t^2}\{G_\mu(q,t)\}$ by L'H$\hat{o}$pital's rule  as follows:
\begin{eqnarray*}
\lim_{t\rightarrow q^+}\frac{\partial^2}{\partial t^2}\{G_\mu(q,t)\}=\lim_{t\rightarrow q^+}\frac{\frac{\partial^3}{\partial t^3}\big\{ A \big\} }{\frac{\partial^3}{\partial t^3}\big\{ [\Gamma_\mu(t)-\Gamma_\mu(q)]^3 \big\} }=\frac{ \Gamma'''_\mu(q^+) }{3\Gamma_\mu'(q)}.
\end{eqnarray*}
To prove item (2), we approximate $\frac{ \Gamma'''_\mu(q^+) }{3\Gamma_\mu'(q)}$ by $\frac{ \Gamma'''_{\delta_0}(0^+) }{3\Gamma_{\delta_0}'(0)}$.  We first compute $\Gamma'''_\mu(q^+)$ by Theorem \ref{thmcal}. For $k=2$ and $L(y_1,y_2)=y^2_2$, we have that
\begin{eqnarray*}
\Gamma''_\mu(u)
&=&2\beta^4 \mathbb{E} \Big[ \big(  2(\partial^{3}_x \Phi_\mu )^2 -4m( \partial^2_x\Phi_\mu)^3   \big) \exp W_\mu(u)\Big],
\end{eqnarray*}
for $u \in (q,q').$
Also  it yields that for $k=3$, $L(y_1,y_2,y_3)=y^2_3$,
\begin{eqnarray*}
 \mathbb{E} \big[  (\partial^{3}_x \Phi_\mu )^2  \exp W_\mu(u)\big] &=&\beta^2 \mathbb{E} \Big[ \Big( 2 (\partial^{4}_x \Phi_\mu)^2  -12m \partial^{2}_x \Phi_\mu  ( \partial^3_x\Phi_\mu)^2  \Big) \exp W_\mu(u)\Big],
\end{eqnarray*}
and for $k=2$, $L(y_1,y_2)=y^3_2$, 
\begin{eqnarray*}
\mathbb{E} \big[  ( \partial^2_x\Phi_\mu)^3   \exp W_\mu(u)\big]&=& \beta^2 \mathbb{E} \Big[ \Big(  6 \partial^2_x\Phi_\mu (\partial^{3}_x \Phi_\mu)^2  -6m (\partial^2_x\Phi_\mu)^4   \Big) \exp W_\mu(u)\Big],
\end{eqnarray*}
for $u \in (q,q')$,
which implies that for $u \in (q,q')$,
\begin{eqnarray*}
\Gamma'''_\mu(u)=8\beta^6 \mathbb{E}\Big[ \Big( (\partial^{4}_x \Phi_\mu)^2 -12m \partial^{2}_x \Phi_\mu  ( \partial^3_x\Phi_\mu)^2 +6m^2   (\partial^2_x\Phi_\mu)^4 \Big) \exp W_\mu(u)  \Big].
\end{eqnarray*}
When $\mu=\delta_0$ and $q=0,$ recall that $W_\mu(0)=0$, $\Gamma'_\mu(0)=2\beta^2$ and $\partial^2_x \Phi_\mu(x,u)=\cosh^{-2}(x)$, for $x \in \mathbb{R}$. We then have that
\begin{eqnarray}\label{square}
\Gamma_{\delta_0}'''(0^+)&=&8\beta^6 \Big(  \big( \partial^{4}_x \Phi_\mu(0,0)\big)^2 +6m^2   \big(\partial^2_x\Phi_\mu(0,0) \big)^4 \Big) \nonumber\\
&=&8 \beta^6 \big( (-2 )^2+6m^2 \big)\nonumber\\
&=& 16 \beta^6 (2+3m^2)\nonumber\\
&>& 32\beta^6.
\end{eqnarray}
Here we use the fact in Lemma 14.7.16 \cite{TalagrandVolI} that $\partial^3_x\Phi_\mu(x,u)$ is odd with respect to $x$ for any $\mu\in M[0,1]$ and $u \in [0,1].$
%and for $\beta \geq \frac{1}{\sqrt{2}},$
%\begin{eqnarray*}
%\lim_{t\rightarrow 0^+}\frac{\partial^2}{\partial t^2}\{G_{\delta_0}(0,t)\}= \frac{16 \beta^6 (2+3m^2)}{3 \cdot 2\beta^2}= \frac{8 \beta^4 (2+3m^2)}{3} \geq  \frac{4}{3} .
%\end{eqnarray*}

Since $\lim_{\beta \rightarrow \frac{1}{\sqrt{2}}} \mu_\beta = \delta_0$ {(in the sense of Lemma \ref{continuitylem})},  by Proposition 1(ii) and Lemma 2 in \cite{AChen15PTRF},   there exists $\eta_0>0$, such that if $q \in [0,\eta_0],$ then $\text{for } \beta \in \big[ \frac{1}{\sqrt{2}}, \frac{1}{\sqrt{2}} +\eta_0],$
\begin{eqnarray*}
\big| \Gamma'_{\mu_\beta}(q) - \Gamma'_{\delta_0}(q) \big| \leq \frac16 \text{ , }
\big| \Gamma'_{\delta_0}(q) - \Gamma'_{\delta_0}(0) \big| \leq \frac16 ,
\end{eqnarray*}
and
\begin{eqnarray*}
\big| \Gamma'''_{\mu_\beta}(q^+) - \Gamma'''_{\delta_0}(q^+) \big| \leq\frac16\text{ , }
\big| \Gamma'''_{\delta_0}(q^+) - \Gamma'''_{\delta_0}(0^+) \big| \leq\frac16,
\end{eqnarray*}
which implies that 
\begin{eqnarray*}
\big| \Gamma'_{\mu_\beta}(q) - \Gamma'_{\delta_0}(0) \big| \leq \frac13 \text{ and }
\big| \Gamma'''_{\mu_\beta}(q^+) - \Gamma'''_{\delta_0}(0^+) \big| \leq \frac13 .
\end{eqnarray*}
%$\text{for } \beta \in \big[ \frac{1}{\sqrt{2}}, \frac{1}{\sqrt{2}} +\eta_2).$ 
For $\beta \in \big[\frac{1}{\sqrt{2}}, \frac{1}{\sqrt{2}}+\eta_0\big]$, we then have that if $q \in [0,\eta_0],$
\begin{eqnarray*}
\Gamma'_{\mu_\beta}(q)  \leq \Gamma'_{\delta_0}(0)  + \big| \Gamma'_{\mu_\beta}(q) - \Gamma'_{\delta_0}(0) \big| \leq 2\beta^2+\frac13,
\end{eqnarray*}
and
\begin{eqnarray*}
\Gamma'''_{\mu_\beta}(q^+)  \geq \Gamma'''_{\delta_0}(0^+)  - \big| \Gamma'''_{\mu_\beta}(q^+) - \Gamma'''_{\delta_0}(0^+) \big| > 32 \beta^6  -  \frac13 .
\end{eqnarray*}
%$\text{for } \beta \in \big[ \frac{1}{\sqrt{2}}, \frac{1}{\sqrt{2}} +\eta_2).$
Therefore we obtain that 
\begin{eqnarray*}
\lim_{t\rightarrow q^+}\frac{\partial^2}{\partial t^2}\big\{G_\mu(q,t)\big\}=\frac{ \Gamma'''_\mu(q^+) }{3\Gamma_\mu'(q)} \geq \frac{32\beta^6-\frac13}{3 (2\beta^2+\frac13)} > \frac12,
\end{eqnarray*}
$\text{for} $ $q \in [0,\eta_0]$ and  $\beta \in \big[\frac{1}{\sqrt{2}},\frac{1}{\sqrt{2}}+\eta_0\big]$.

Finally we show that $ \frac{\partial^3}{\partial t^3} \{G_{\mu}(q,t)\}$ is bounded for  $t \in (q,q')$. We compute $\frac{\partial^3}{\partial t^3} \big\{ G_\mu(q,t) \big\}$ explicitly as follows:
\begin{eqnarray*}
\frac{\partial^3}{\partial t^3} \big\{ G_\mu(q,t) \big\} &=& \frac{ [\Gamma_\mu(t)-\Gamma_\mu(q)]  \cdot  \frac{\partial}{\partial t} \{ A \} - 3\Gamma'_\mu(t) \cdot A}{[\Gamma_\mu(t)-\Gamma_\mu(q)]^4}.
\end{eqnarray*}
Note that 
\begin{eqnarray*}
&&\frac{\partial}{\partial t} \Big\{ [\Gamma_\mu(t)-\Gamma_\mu(q  )]    \frac{\partial}{\partial t} \{ A \} - 3\Gamma'_\mu(t)  A  \Big\}=  [\Gamma_\mu(t)-\Gamma_\mu(q  )]   \frac{\partial^2}{\partial t^2} \{ A \} -2 \Gamma'_\mu(t)   \frac{\partial}{\partial t} \{ A \} -3 \Gamma''_\mu(t)   A,\\
&&\frac{\partial^2}{\partial t^2} \Big\{ [\Gamma_\mu(t)-\Gamma_\mu(q  )]    \frac{\partial}{\partial t} \{ A \} - 3\Gamma'_\mu(t)  A  \Big\}= -  \Gamma'_\mu(t)   \frac{\partial^2}{\partial t^2} \{ A \}  -3 \Gamma'''_\mu(t)   A  -5 \Gamma''_\mu(t)   \frac{\partial}{\partial t} \{ A \}\\
&&\qquad \qquad \qquad \qquad \qquad \qquad \qquad \qquad \qquad +  [\Gamma_\mu(t)-\Gamma_\mu(q  )]   \frac{\partial^3}{\partial t^3} \{ A \} ,\\
&&\frac{\partial^3}{\partial t^3} \Big\{ [\Gamma_\mu(t)-\Gamma_\mu(q  )]    \frac{\partial}{\partial t} \{ A \} - 3\Gamma'_\mu(t)  A  \Big\}= -  6\Gamma''_\mu(t)   \frac{\partial^2}{\partial t^2} \{ A \}  -8 \Gamma'''_\mu(t)   \frac{\partial}{\partial t} \{ A \}  -3 \Gamma^{(4)}_\mu(t)   A \\
&&\qquad \qquad \qquad \qquad \qquad \qquad \qquad \qquad \qquad  +  [\Gamma_\mu(t)-\Gamma_\mu(q  )]   \frac{\partial^4}{\partial t^4} \{ A \}  , \\
&&\frac{\partial^4}{\partial t^4} \Big\{ [\Gamma_\mu(t)-\Gamma_\mu(q  )]    \frac{\partial}{\partial t} \{ A \} - 3\Gamma'_\mu(t)  A  \Big\}= -  14\Gamma'''_\mu(t)   \frac{\partial^2}{\partial t^2} \{ A \}  -11 \Gamma^{(4)}_\mu(t)   \frac{\partial}{\partial t} \{ A \}  -3 \Gamma^{(5)}_\mu(t)   A \\
&& \qquad \qquad \qquad \qquad \qquad \qquad \qquad +  [\Gamma_\mu(t)-\Gamma_\mu(q  )]   \frac{\partial^5}{\partial t^5} \{ A \} -  6\Gamma''_\mu(t)   \frac{\partial^3}{\partial t^3} \{ A \}    + \Gamma'_\mu(t) \frac{\partial^4}{\partial t^4} \{ A \}  .
\end{eqnarray*}
When $t \rightarrow q^+$, by \eqref{A0},
%\eqref{q120} 
%and L'H$\hat{o}$pital's rule, 
we have that 
\begin{eqnarray*}
\lim_{t \rightarrow q^+}\frac{\partial^k}{\partial t^k} \Big\{ [\Gamma_\mu(t)-\Gamma_\mu(q)]    \frac{\partial}{\partial t} \{ A \} - 3\Gamma'_\mu(t)  A  \Big\}=0, 
\end{eqnarray*}
and
\begin{eqnarray*}
\lim_{t \rightarrow q^+}\frac{\partial^k}{\partial t^k} \Big\{ [\Gamma_\mu(t)-\Gamma_\mu(q)]^4  \Big\}=0,
\end{eqnarray*}
$\text{for } k=0,1,2,3.$
Then by L'H$\hat{o}$pital's rule, 
\begin{eqnarray*}
\lim_{t \rightarrow q  ^+}\frac{\partial^3}{\partial t^3} \big\{ G_\mu(q  ,t) \big\} &=&  \lim_{t \rightarrow q  ^+}\frac{ \frac{\partial^4}{\partial t^4 } \big\{   [\Gamma_\mu(t)-\Gamma_\mu(q  )]  \cdot  \frac{\partial}{\partial t} \{ A \} - 3\Gamma'_\mu(t) \cdot A \big\} }{  \frac{\partial^4}{\partial t^4 } \big\{ [\Gamma_\mu(t)-\Gamma_\mu(q  )]^4 \big\}}\\
&=& \frac{-  12 \Gamma_\mu'(q  )^2  \Gamma''_\mu(q^+  )   \Gamma'''_\mu(q^+  )  + \Gamma'_\mu(q^+  ) \frac{\partial^4}{\partial t^4} \{ A \}_{|t=q^+  } }{24 \Gamma'_\mu(q  )^4}. 
\end{eqnarray*}
%For any $t \in [q_1,1]\setminus $supp $\mu$, $\frac{\partial^3}{\partial t^3} \big\{G_\mu(q_1,t)\}$ exists and there exists a constant $D>0$ independent of $\beta$, $q_1$, $q_2$ and $t$ such that  $\big| \frac{\partial^3}{\partial t^3} \{G_{\mu}(q_1,t)\}\big|\leq D$.
%By Proposition 5 in \cite{AChen15PTRF}, for $j \geq 1,$ there exists constants $C_{0,j}>0$ such that 
%$$\sup_{\mathbb{R} \times [0,1]} |\partial^j_x \Phi_\mu | \leq C_{0,j}$$
%which implies that, with the fact $\mathbb{E}[\exp W_\mu(u)] =1$, $u \in [0,1],$
%\begin{eqnarray*}
%\big|\mathbb{E}[\partial^j_x \Phi_\mu (M(u),u) \exp W_\mu(u)] \big| \leq  C_{0,j},
%\end{eqnarray*}
%for $u \in [0,1]$ and $j \geq 1.$
By Proposition 1 in \cite{AChen15PTRF} and Theorem \ref{thmcal}, for $\beta \in [\frac{1}{\sqrt{2}},100],$ there exists a constant  $M>0$ such that  $\big|\Gamma_\mu'(q)\big|$, $\big|\Gamma_\mu''(q^+)\big|$, $\big|\Gamma_\mu'''(q^+)\big|$ and $\big|\frac{\partial^4}{\partial t^4} \{ A \}_{|t=q^+}\big|$ are all bounded by $D_1,$ which implies that 
\begin{eqnarray*}
\Big| \lim_{t \rightarrow q^+}\frac{\partial^3}{\partial t^3} \big\{ G_\mu(q,t) \big\} \Big| <D_2,
\end{eqnarray*}
for some $D_2>0.$ 
%By the definition of $\Gamma'_\mu(u),$ we have that $\Gamma'_\mu(u)>0$ for $u \in (0,1),$ which implies that . 
Since $\Gamma_{\mu}(u) \neq \Gamma_{\mu}(q)$ for $u \in [0,1] \setminus \{q\}$ and any function of the form \eqref{Fmu} is a continuous function and uniformly on $[0,1],$ 
then by Theorem \ref{thmcal}, $\frac{\partial^3}{\partial t^3} \big\{ G_\mu(q,t) \big\} $ is  bounded for $t \in (q,q')$.

\end{proof}

Finally, we prove the three inequalities \eqref{GIP1}, \eqref{GIP2} and \eqref{GIP3} in the proof of Theorem \ref{Fm}.
\begin{lemma}\label{lemmaGIP}
The three inequalities \eqref{GIP1}, \eqref{GIP2} and \eqref{GIP3} hold for any $x\in \mathbb{R}\setminus \{0\} $.
\end{lemma}
\begin{proof}[Proof of Lemma \ref{lemmaGIP}]
Note that we can write $b_n(x)$ for $n=-1,-3,-5$ explicitly as follows:
\begin{eqnarray*}
&&b_{-1}(x)=\arctan(\sinh(x)), \\
&&b_{-3}(x)=\frac{1}{2}\big[ \arctan(\sinh(x))+\sinh(x) \cosh^{-2}(x) \big],\\
&&b_{-5}(x)=\frac{1}{8}\big[ 3\arctan(\sinh(x))+2\sinh(x) \cosh^{-4}(x)+3\sinh(x)\cosh^{-2}(x)\big].
\end{eqnarray*}
Since the three functions in \eqref{GIP1}, \eqref{GIP2}, \eqref{GIP3} are all even functions, we will only prove the three inequalities for $x >0.$

We first prove \eqref{GIP1} by showing that 
\begin{eqnarray}\label{g1x}
g_1(x):=\cosh( x )- \cosh^{-1}( x ) - x \cdot   b_{-1}( x ) 
\end{eqnarray}
is strictly increasing for $x>0.$ We compute the derivative of $g_1(x)$ as follows:
\begin{eqnarray*}
&&\frac{d}{dx}\big\{ g_1(x)  \big\}=\sinh(x) \cdot \bigg[ 1+\frac{1}{\cosh^2(x)}-\frac{x }{\sinh(x) \cosh(x)}-\frac{\arctan(\sinh(x))}{\sinh(x)} \bigg].
\end{eqnarray*}
Based on the inequality $\frac{\arctan(z)}{z} \leq \frac{1}{(1+z^2)^{1/3}}$ for $z \in \mathbb{R}$ in \cite{arcsinh}, we set $z= \sinh x$. We then obtain that, for $x >0,$
\begin{eqnarray*}
\frac{d}{dx}\big\{ g_1(x)  \big\}\geq \sinh(x) \cdot \big(1+\frac{1}{\cosh^2(x)}-\frac{x }{\sinh(x) \cosh(x)}-\frac{1}{\cosh^{2/3}(x)}\big).
\end{eqnarray*}
Now let $t=\cosh(x)$,  in order to show \eqref{GIP1}, it then suffices for us to show that for $t>1,$
\begin{eqnarray}\label{eqgip1}
1+\frac{1}{t^2}-\frac{\text{ arccosh}(t)}{t\sqrt{t^2-1}}-t^{-2/3}>0.
\end{eqnarray}
We reformulate \eqref{eqgip1} as follows:
\begin{eqnarray}\label{ieqgip1}
\big(t+\frac{1}{t}-t^{1/3}\big)\sqrt{t^2-1} - \text{arccosh}(t)>0, \text{ for }t>1.
\end{eqnarray}
%In order to show \eqref{GIP1}, it then suffices for us to show that for $x>0.$
%\begin{eqnarray}\label{eqgip1}
%1+\frac{1}{\cosh^2(x)}-\frac{x }{\sinh(x) \cosh(x)}-\frac{1}{\cosh^{2/3}(x)}>0.
%\end{eqnarray}
%Now let $t=\cosh(x)$, and then \eqref{eqgip1} holds for $x>0$ is equivalent to the following holds for $t >1,$
In order to prove \eqref{ieqgip1}, we show that the left hand side of \eqref{ieqgip1} is strictly increasing by computing its derivative as follows:
\begin{eqnarray*}
&&\frac{d}{dt}\big\{ \big(t+\frac{1}{t}-t^{1/3}\big)\sqrt{t^2-1} - \text{arccosh}(t) \big\}\\&=&\frac{6t^4-4t^{10/3}-6t^2+t^{4/3}+3}{3t^2\sqrt{t^2-1}}\\
&=& \frac{(t^{1/3}-1)^2}{3t^2\sqrt{t^2-1}}\cdot \big( 6t^{10/3}+12t^3+14t^{8/3}+16t^{7/3}+18t^2
+20t^{5/3}+16t^{4/3}+12t+9t^{2/3}+6t^{1/3}+3 \big)\\
&>&0, \text{ for }t>1,
\end{eqnarray*}
which finish our proof for \eqref{GIP1}.

We then prove \eqref{GIP2} by showing that 
\begin{eqnarray}\label{g2x}
g_2(x):=2\cosh( x )- \cosh^{-1}( x )- \cosh^{-3}( x ) -3 x \cdot   b_{-3}( x ) 
\end{eqnarray}
is strictly increasing for $x>0.$ We compute the derivative of $g_2(x)$ as follows:
\begin{eqnarray*}
\frac{d}{dx}\big\{ g_2(x)  \big\}&=&\sinh(x) \cdot \bigg[ 2-\frac{3\arctan(\sinh(x))}{2 \sinh(x)}-\frac{1}{2(\cosh(x))^2}+ \frac{3}{ \cosh^4(x)} -\frac{3x}{\sinh(x)\cosh^3(x)} \bigg].
\end{eqnarray*}
Similarly, it suffices for us to show that
\begin{eqnarray}\label{ieqgip2}
\Big(-\frac{3}{2}t^{7/3}+2t^3-\frac{1}{2}t+3t^{-1}\Big)\sqrt{t^2-1} - 3 \text{arccosh}(t)>0, \text{ for }t>1.
\end{eqnarray}
In order to prove \eqref{ieqgip2}, we show that the left hand side of \eqref{ieqgip2} is strictly increasing as follows:
\begin{eqnarray*}
&&\frac{d}{dt}\Big\{ \Big(-\frac{3}{2}t^{7/3}+2t^3-\frac{1}{2}t+3t^{-1}\Big)\sqrt{t^2-1} - 3 \text{arccosh}(t) \Big\}\\&=&\frac{16t^6-10t^{16/3}-14t^4+7t^{10/3}-5t^2+6}{2t^2\sqrt{t^2-1}}\\
&=& \frac{(t^{1/3}-1)^2}{2t^2\sqrt{t^2-1}}\cdot \big( 16t^{16/3}+32t^5+38t^{14/3}+44t^{13/3}+50t^4+56t^{11/3}+48t^{10/3}\\
&&+40t^3+39t^{8/3}+38t^{7/3}+37t^2
+36t^{5/3}+30t^{4/3}+24t+18t^{2/3}+12t^{1/3}+6 \big)\\
&>&0,
\end{eqnarray*}
which finish our proof for \eqref{GIP2}.

Finally we turn to \eqref{GIP3}. By a similar reasoning, it suffices for us to show that
\begin{eqnarray*}
g_3(t):=\frac{4}{3}\cdot\frac{t+5t^{-3}-6t^{-5}}{\sqrt{t^2-1}(t^{-2/3}+6t^{-4}+t^{-2})}-\text{arccosh}(t) >0, \text{ for } t>1.
\end{eqnarray*}
We compute the derivative of $g_3(t)$ as follows:
\begin{eqnarray*}
&&\frac{d}{dt}\{g_3(t)\} = \frac{1}{9t^2 \sqrt{t^2-1}(t^{10/3}+t^2+6)^2} \cdot g_4(t),
\end{eqnarray*}
where $g_4(t):= 8t^{28/3}-9t^{26/3}+24t^8-30t^{22/3}+267t^6-320t^{16/3}-312t^4+312t^{10/3}-180t^2+432.$
By a standard application of Sturm's theorem, it yields that  $g_4(t^3)$ has exactly 2 roots in $(1+\infty).$ Therefore $g_4(t)$ also has exactly 2 roots in $(1+\infty)$ and $g_3(t)$ then has exactly 2 critical points in $(1,+\infty).$ 

Since $g_4(t)>0,$ for $t=1.25$ and $g_4(t)<0$, for $t=1.26$, then $g_3(t)$ has a local maxima point in $(1.25,1.26).$ Since $g_4(t)<0,$ for $t=1.5$ and $g_4(t)>0$, for $t=1.51$, then $g_3(t)$ has the other critical point in $(1.5,1.51),$ which is a local minima point. We denote this unique local minima point of $g_3(t)$ in $(1,+\infty)$ by $t_m.$ 
Then it holds that
$$g_3(t_m) \geq \frac{4}{3}\cdot\frac{1.25+5\cdot1.25^{-3}-6\cdot1.25^{-5}}{\sqrt{1.25^2-1}(1.25^{-2/3}+6\cdot 1.25^{-4}+1.25^{-2})} - \text{arccosh} (1.26) >0.$$
Here we use the fact that both $\text{arccosh}(t)$  and $\frac{4}{3}\cdot\frac{t+5t^{-3}-6t^{-5}}{\sqrt{t^2-1}(t^{-2/3}+6t^{-4}+t^{-2})}$ are increasing for $t \in(1,+\infty).$ Indeed, the derivative of $\frac{t+5t^{-3}-6t^{-5}}{\sqrt{t^2-1}(t^{-2/3}+6t^{-4}+t^{-2})}$ is strictly positive for $t>1$:
\begin{eqnarray*}
\frac{d}{dt} \Big\{ \frac{t+5t^{-3}-6t^{-5}}{\sqrt{t^2-1}(t^{-2/3}+6t^{-4}+t^{-2})} \Big\} = \frac{1}{3t^2\sqrt{t^2-1}(t^{10/3}+t^2+6)^2}\cdot g_5(t) 
\end{eqnarray*}
where $g_5(t):=2t^{28/3}+6t^8-3t^{22/3}+69t^6-53t^{16/3}-51t^4+78t^{10/3}+36t^2+108.$ By a standard application of Sturm's theorem, we can find that  $g_5(t^{2/3})$ has no roots in $[0,+\infty)$, which verifies our claim. 
Since $\lim_{t \rightarrow 1^+}g_3(t)=0$ and $\lim_{t \rightarrow +\infty}g_3(t)=+\infty,$ we then conclude that $g_3(t) > 0$ for $t \in (1,+\infty).$

\end{proof}

%\bibliographystyle{abbrv}
%\bibliography{biblio3}

\end{document}